\documentclass[a4paper,12pt]{article}
\usepackage{amsmath}
\usepackage[vlined]{algorithm2e}
\usepackage{diagbox}
\usepackage{enumerate}
\usepackage{appendix}
\usepackage{longtable}
\usepackage{cases}
\usepackage{tikz}
\usepackage{caption}
\usetikzlibrary{arrows}
\usetikzlibrary{shapes,calc}
\usepackage{slashbox}

\SetFuncSty{textsc}
        \SetKwFunction{frun}{Run}
        \SetKwHangingKw{arun}{\frun}
        \SetKwFunction{fset}{Set}
        \SetKwHangingKw{aset}{\fset}
        \SetKwFunction{fselect}{Select}
        \SetKwHangingKw{aselect}{\fselect}
        \SetKwFunction{fcompute}{Compute}
        \SetKwHangingKw{acompute}{\fcompute}
        \SetKwFunction{fsolve}{Solve}
        \SetKwHangingKw{asolve}{\fsolve}
        \SetKwFunction{festimate}{Estimate}
        \SetKwHangingKw{aestimate}{\festimate}
        \SetKwFunction{fmark}{Mark}
        \SetKwHangingKw{amark}{\fmark}
        \SetKwFunction{frefine}{Refine}
        \SetKwHangingKw{arefine}{\frefine}
        \SetKwFunction{compute}{Compute}
        \SetKwFunction{set}{Set}

{\algorithm}%
   {\endalgorithm}
\newcommand{\cR}{\mathcal{R}}

\usepackage{mathtools}

\DeclareFontEncoding{LS1}{}{}
\DeclareFontSubstitution{LS1}{stix}{m}{n}
\DeclareSymbolFont{symbols2}{LS1}{stixfrak} {m} {n}
\DeclareMathSymbol{\operp}{\mathbin}{symbols2}{"A8}

\title{A~priori and a~posteriori error analysis of the Crouzeix-Raviart and Morley FEM with original and modified right-hand sides}
\author{Carsten Carstensen\footnote{Department of Mathematics, 
Humboldt-Universit\"{a}t zu Berlin, 10099 Berlin, Germany.
Distinguished Visiting Professor, Department of Mathematics, Indian institute of 
Technology Bombay, Powai, Mumbai-400076, India.  cc@math.hu-berlin.de}        
\quad\text{and}\quad 
Neela Nataraj\footnote{Department of Mathematics, Indian Institute of Technology Bombay, Powai, Mumbai 400076, India. neela@math.iitb.ac.in}
}
\usepackage{amsmath,amsthm,amssymb,enumerate}
\usepackage[T1]{fontenc}
\usepackage{gensymb}
\usepackage[square,sort&compress,comma,numbers]{natbib}
\usepackage[sc]{mathpazo}
\usepackage{multirow} 
\usepackage{newtxtext,newtxmath}
\usepackage{subfig}
\usepackage{graphicx}
\usepackage{epstopdf}
\usepackage{cancel}
\usepackage[margin=3cm]{geometry}
\usepackage{tikz}
\usepackage{caption}
\usetikzlibrary{shapes,calc}
\usepackage{verbatim}
\usepackage{mathrsfs}
\usepackage{accents}
\usepackage[utf8]{inputenc}

\usetikzlibrary{decorations.pathmorphing}
\usetikzlibrary{decorations.pathreplacing}
\usetikzlibrary{positioning}
\usetikzlibrary{shapes}
\usetikzlibrary{arrows}
\usetikzlibrary{patterns}
\usetikzlibrary{fadings}
\usetikzlibrary{plotmarks}
\usetikzlibrary{calc}
\usetikzlibrary{intersections}
\tikzstyle{every picture}+=[font=\footnotesize]
\usepackage{paralist}
\usepackage{bbm}
\usepackage{latexsym}           
\usepackage{enumerate}
\usepackage{enumitem}

\setlist{noitemsep, topsep=0.8ex, partopsep=0pt
	, leftmargin=3em}
\setlist[1]{labelindent=\parindent}

\newlist{axioms}{enumerate}{1}
\setlist[axioms]{font=\bfseries}

\newlist{alphenum}{enumerate}{1}
\setlist[alphenum]{label=\textbf{(\alph*)}, leftmargin=4em}

\newlist{alphienum}{enumerate}{1}
\setlist[alphienum]{label=\textit{(\alph*)}}

\newlist{romanenum}{enumerate}{1}
\setlist[romanenum]{label=\textit{(\roman*)}}

\newlist{romaninenum}{enumerate*}{1}
\setlist[romaninenum]{label=\textit{(\roman*)}}
\usepackage[vlined]{algorithm2e}
\SetKwIF{If}{ElseIf}{Else}{if}{}{else if}{else}{endif}
\SetKwFor{For}{for}{}{endfor}
\usepackage[noabbrev, capitalise]{cleveref}
\usepackage{tabu}
\crefname{equation}{\unskip}{\unskip}
\creflabelformat{equation}{#2(#1)#3}
        \SetFuncSty{textsc}
        \SetKwFunction{frun}{Run}
        \SetKwHangingKw{arun}{\frun}
        \SetKwFunction{fset}{Set}
        \SetKwHangingKw{aset}{\fset}
        \SetKwFunction{fselect}{Select}
        \SetKwHangingKw{aselect}{\fselect}
        \SetKwFunction{fcompute}{Compute}
        \SetKwHangingKw{acompute}{\fcompute}
        \SetKwFunction{fsolve}{Solve}
        \SetKwHangingKw{asolve}{\fsolve}
        \SetKwFunction{festimate}{Estimate}
        \SetKwHangingKw{aestimate}{\festimate}
        \SetKwFunction{fmark}{Mark}
        \SetKwHangingKw{amark}{\fmark}
        \SetKwFunction{frefine}{Refine}
        \SetKwHangingKw{arefine}{\frefine}
        \SetKwFunction{compute}{Compute}
        \SetKwFunction{set}{Set}
\newtheorem{thm}{Theorem}[section]

\newtheorem{lem}[thm]{Lemma}
\newtheorem{prop}[thm]{Proposition}

\theoremstyle{definition}
\newtheorem{defn}{Definition}[section]

\newtheorem{rem}{Remark}[section]
\newtheorem{example}{\bf Example}[section]

\numberwithin{equation}{section}

\newcommand{\cE}{\mathcal E}

\newcommand{\cT}{\mathcal T}

\newcommand{\pw}{\mathrm{pw}}
\newcommand{\nc}{\mathrm{nc}}
\newcommand{\jc}{\mathrm{J}}
\newcommand{\osc}{\mathrm{osc}}
\newcommand{\qo}{\mathrm{qo}}

\newcommand{\NC}{\text{pw}}
\newcommand{\C}{\mathrm{C}}

\newcommand{\trinl}{\ensuremath{|\!|\!|}}
\newcommand{\trinr}{\ensuremath{|\!|\!|}}

\newcommand{\dx}{{\rm\,dx}}
\newcommand{\dy}{{\rm\,dy}}
\newcommand{\dz}{{\rm\,dz}}
\newcommand{\ds}{{\rm\,ds}}
\newcommand{\dw}{{\rm\,dw}}
\newcommand{\dt}{{\rm\,dt}}
\newcommand{\dr}{{\rm\,dr}}

\DeclareMathOperator{\E}{\mathcal{E}}
\newcommand{\M}{\mathrm{M}}
\newcommand{\CR}{\mathrm{CR}}

\newcommand{\T}{\mathcal{T}}

\newcommand{\V}{\mathcal{V}}

   \usepackage{xparse}
   
   \newcounter{const}
\NewDocumentCommand{\constant}{o}
 {
  \IfValueTF{#1}%
  {C_{#1}}%
  {\refstepcounter{const}%
  C_{\theconst}}%
 }

\usepackage{tikz,pgfplots}
\usetikzlibrary{calc,angles,positioning,intersections,quotes,decorations.markings}
\usepackage{tkz-euclide}
\pgfplotsset{compat=1.5}

\begin{document}	
\maketitle

\begin{abstract}
This article on  nonconforming schemes for $m$ harmonic problems 
simultaneously treats the Crouzeix-Raviart ($m=1$) and the 
Morley finite elements  ($m=2$) for the original and for modified right-hand side $F$ 
in the dual space $V^*:=H^{-m}(\Omega)$ to the energy space $V:=H^{m}_0(\Omega)$.
The smoother $J:V_\nc\to V$ in this paper is a companion operator, that is a 
linear and bounded right-inverse  to the nonconforming interpolation operator $I_\nc:V\to V_\nc$,  and modifies the discrete right-hand side $F_h:=F\circ J \in V_\nc^*$. The 
best-approximation property of the modified scheme from Veeser et~al. (2018)
is recovered and complemented with an analysis of the  convergence rates 
in weaker Sobolev norms. Examples with oscillating data 
show that the original method may fail to 
enjoy the best-approximation property but can also be better than the 
modified scheme. The  a~posteriori analysis of this paper concerns data oscillations 
of various types in a class of right-hand sides $F\in V^*$. The reliable error estimates 
involve explicit constants and can be recommended for explicit error control of the piecewise energy norm. The efficiency follows solely up to data oscillations and examples 
illustrate  this can be problematic. 
\end{abstract}

\noindent {\bf Keywords:} nonconforming schemes, Crouzeix-Raviart, Morley, biharmonic problem, medius error analysis, best-approximation, a~priori and a~posteriori analysis

\noindent {\bf AMS Classification:}  65N30, 65N12, 65N50

\section{Introduction}
The beginning of the nonconforming finite element methodology, its motivating examples in fluid mechanics, and its 
connections with mixed formulations are outlined in \cite{Brenner2015} together with the relevant literature that (partly) led to the 
Strang lemmas (also named after Strang-Fix and Strang-Berger etc.) in  
finite element textbooks. 
This article aims at a simple a~priori and a~posteriori error analysis for $m$ harmonic problems, that involve 
the energy space $V:= H^m_0(\Omega)$ endowed with the energy scalar product $a: V\times V\to\mathbb{R}$. The 
emphasis is on the lowest-order  nonconforming schemes with  the piecewise polynomial,  
nonconforming discrete space $V_\nc\subset P_m(\T)$, 
that simultaneously describes  the Crouzeix-Raviart (CR) for $m=1$ and Morley finite element methods (FEM) for $m=2$.
Those nonconforming FEMs allow for an interpolation operator $I_\nc: V\to V_\nc$ with the best-approximation property
and interpolation error estimates with explicit (and small) constants  (called $\kappa_m$ throughout) from \cite{CarstensenGedickeRim,CarstensenGedicke2014} for 
$m=1$ and \cite{ccdg_2014} for $m=2$.
The word Strang lemma is no longer used in this paper  and a {\em companion operator}  $J:V_\nc\to V$ 
circumvents the medius analysis from \cite{Gudi10}.

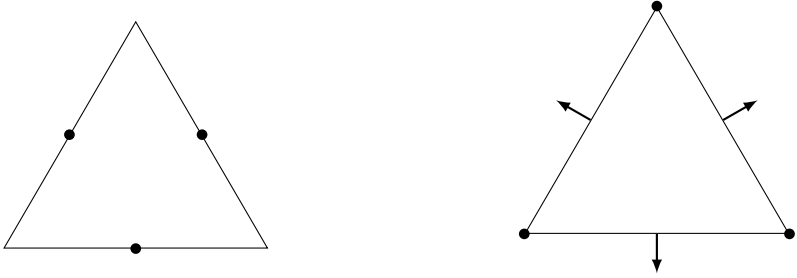
\begin{figure*}[h]
\begin{minipage}[h]{0.45\linewidth}
\begin{center}
\begin{tikzpicture}
\node[regular polygon, regular polygon sides=3, draw, minimum size=4cm]
(m) at (0,0) {};


\put(-3,-5){$T$}



\fill [black] (m.side 1) circle (2pt);
\fill [black] (m.side 2) circle (2pt);
   \fill [black] (m.side 3) circle (2pt);
\end{tikzpicture}
\end{center}
\end{minipage}
\begin{minipage}[h]{0.45\linewidth}
\begin{center}
\begin{tikzpicture}
\node[regular polygon, regular polygon sides=3, draw, minimum size=4cm]
(m) at (0,0) {};

\fill [black] (m.corner 1) circle (2pt);

\put(-3,-5){$T$}

\fill [black] (m.corner 2) circle (2pt);
\fill [black] (m.corner 3) circle (2pt);

\draw [-latex, thick] (m.side 1) -- ($(m.side 1)!0.3!90:(m.corner 1)$);
\draw [-latex, thick] (m.side 2) -- ($(m.side 2)!0.3!90:(m.corner 2)$);
\draw [-latex, thick] (m.side 3) -- ($(m.side 3)!0.3!90:(m.corner 3)$);
\end{tikzpicture}
\end{center}
\end{minipage}
\caption{Crouzeix-Raviart FEM (left) and Morley FEM (right)}\label{fig.ncfem}
\end{figure*}

A  {\it (conforming) companion operator} $J:V_\nc\to V$  is defined in this paper as a right-inverse 
of the interpolation operator $I_\nc: V\to V_\nc$ with 
additional benefits such as the $L^2$ orthogonality $ v_\nc-J v_\nc \perp P_m(\T)$  for all $v_\nc\in V_\nc$. 
The extra $L^2$ orthogonality is {\em not } needed for the sole  best-approximation \cite{veeser_zanotti1,veeser_zanotti2, veeser_zanotti3}, but it
enables a duality argument for error estimates in weaker Sobolev norms  \cite{ccdgmsMathComp}. 

\medskip 
The paper departs with a revisit of  the medius analysis \cite{Gudi10}, the best-approximation results from  \cite{veeser_zanotti1,veeser_zanotti2, veeser_zanotti3} and recovers the same 
best-approximation constant $C_{\rm qo}$ for a modified right-hand side $F _h$ in Section~\ref{section:best_approximation}.~It thereafter contributes a comprehensive  analysis of convergence rates in  Sobolev  
norms weaker and stronger than the energy norm in $V$ in Section~\ref{sec:lower_order}.

\medskip
The precise convergence rate  $\min\{ 2\sigma, m+\sigma-s\}$ is guaranteed  for quasi-uniform meshes  for the error $u-Ju_\nc$ of the post-processed approximation $J u_\nc$ in the norm of $H^{s}(\Omega)$  for $-m\le s < m+{1/2}$ with the positive index  $\sigma\le 1$ in terms of elliptic regularity. This holds under the generic assumption that the convergence rate in the energy norm is exactly $\sigma<m$ and so excludes untypical smooth solutions.  The only former result is the seminal contribution \cite{hujunshi2012} with a proof that the convergence rate in the $L^2$ error in nonconforming schemes is no better than $2\sigma$ not only for one model scenario but for many other nonconforming schemes.

\medskip

The paper will focus on shape-regular triangulations of the planar bounded polygonal Lipschitz domain 
$\Omega\subset {\mathbb R}^2$ into triangles for simplicity. The abstract analysis 
holds verbatim for any space dimensions once the companion operator  $J$  is designed. This started in 2D  for $m=1$ 
with \cite{ccdgmsMathComp} for the  Crouzeix-Raviart finite element method  (CRFEM) and is almost immediately generalized to any space dimension.
For $m=2$ and the  Morley FEM  it started in the PhD thesis \cite{Gallistl2014} and details can be found in 
\cite{DG_Morley_Eigen,veeser_zanotti2} in 2D, while the companion operator 
$J$ for generalization of the Morley FEM  from \cite{MingXu}   is designed  in \cite{CCP_new} in 3D. 
The most refined analysis in 2D with explicit constants and the connection to a~posteriori error estimates are summarized in  \cite{CH17,CCP}.
The paper will focus on lowest-order schemes because, first, the interpolation operators are less clear for 
higher-order CRFEM in 2D and 3D \cite{Sauter2018}, hence, second, conforming companions are unknown. In fact, there are two aspects that make the
nonconforming FEM attractive in comparison to conforming FEM: The interpolation operator is local (simplex-oriented) and defined for any function in  the energy space.  
Those advantages are possible for simplicial meshes in any space dimension, but {\em not} available for all higher polynomial degrees. 
 
\medskip

The modified nonconforming FEM specifies the discrete right-hand side 
 $F_h:= F\circ J\in V_\nc^*$ with the companion operator $J:V_\nc\to V$ as a smoother. Thereby the modified version
allows for general data $F\in V^*$. Modifications of the right-hand side $F_h$ of this type had been introduced  in \cite{ArnoldBrezzi85,BS05}, before 
 \cite{veeser_zanotti1,veeser_zanotti2, veeser_zanotti3} systematically studied the resulting best-approximation property and established the name smoother.
The modified FEM is a Petrov-Galerkin scheme discussed in Section~\ref{section:best_approximation} and hence quasi-optimally convergent (once stable).  
%
If the source term $F \in V^*\setminus  L^2(\Omega) $ is {\em not} a Lebesgue function, the application of the original discrete problem is less obvious.  
This paper studies a class of right-hand sides $F$ in $V^*$ that includes point  forces for $m=2$ in the model example to define $F_h$ (without any smoother) adapted to the data. 
This provokes a comparison of the nonconforming FEM with the two different discrete right-hand sides with advantages for both sides. 
The model example is rich enough to allow for a {\em counterexample for the  best-approximation} of the original method. 
The latter holds up to the data oscillations only;  cf. \cite{CPS2012} for $m=1$ and  $F\in L^2(\Omega)$.    
Conversely, another example illustrates that the modified nonconforming method (with the discrete right-hand side $F\circ J$) can be {\em worse} than the original discretization.   

\medskip

The a~posteriori error control for the (original) nonconforming FEM and $F\in L^2(\Omega)$ started with \cite{DEDRPCVV1996}   and continued with 
\cite{CCSBSJ2002,Carstensen2009,CCHu2007,CCHJOA2007,VR2012,CCEMHRHWLC2012,CCDGJH14} for $m=1$ and 
 \cite{HuShi09,HuShiXu2012, CCDGHU13,  VeigaNiiranenStenberg07}
for $m=2$. The initial work was based on a 
Helmholtz decomposition of the piecewise energy norm of the nonconforming error. The Helmholtz decompositions become Hodge decompositions
for multiply-connected domains and are not all known for 3D. Their discrete counterparts  from  \cite{mindlin} for $m=1$
and \cite{CCDGJH14} for $m=2$ are not known for all domains and space dimensions. That motivates the use of other techniques and  
this paper suggests a paradigm shift: 
Utilize the companion operator $J$ for a split in the conforming error $u- J u_\nc$ and in the nonconforming  error $u_\nc - J u_\nc$. 
This approach leads to a~posteriori error estimates even if a Helmholtz decomposition is unknown and without a piecewise integration by parts.
As a consequence,  the error estimators do {\em not} contain weighted jumps of the normal components of  derivatives of the discrete solution $u_\nc$ across interior edges. 
On the other hand, the efficiency  requires some other additional benefit of the companion operator, which actually motivated its first design in a~posteriori error control. 
%
Rate-optimal  adaptive  nonconforming FEM are analyzed in  \cite{CCP_new,ccdgmsMathComp,BeckerMaoShi2010,CCDGJH14,Carstensen2019,rabus2010}  and the references therein.  Amongst all second-order schemes, the Morley FEM appears to be the most 
 simple and method of choice for fourth-order problems \cite{CCGMNN_Semilinear}.
 Hence the generalization of the  methodology in this article for $m \ge 3$ with nonconforming finite elements from \cite{WMXJ2013} appears appealing.

\medskip
\noindent
{\bf Outline of the structure.}
{The remaining parts of this paper are organised as follows. In Section \ref{section:quasiorthogonality} nonconforming schemes, interpolation and companion operators, and examples are presented. The best-approximation results available in literature are reviewed and an abstract best-approximation result is established for data  $F \in H^{-m}(\Omega)$ in Section~\ref{section:best_approximation}. Section~\ref{sec:lower_order}  presents lower-order error estimates and  precise convergence rates 
 under further hypotheses. 
 Section~\ref{sec:comments} discusses  examples of data in $V^*$, an alternate choice of the right-hand side functional that leads  to a modified scheme and a comparison of the original and the modified scheme. 
 This is followed by  an a~posteriori analysis in Section \ref{sec:aposteriori}.

\medskip \noindent
{\bf General notation.}  Standard notation of  Lebesgue and Sobolev spaces, 
their norms, and $L^2$ scalar products  applies throughout the paper
such as the abbreviation $\|\bullet\|$ for $\|\bullet\|_{L^2(\Omega)}$.
For a  shape-regular triangulation $\T$ of $\Omega$ into triangles, the product space  $\prod_{T\in\T} H^s(T)$  denoted as $H^s(\T)$ is defined as $ \{ v_\pw\in L^2(\Omega): \forall T\in\T, \;   v_\pw|_T\in H^s(T)\}$ and is equipped with the Euclid norm of those contributions $ \|\bullet  \|_{H^s(T)}$ for all $T\in\T$. The triple norm  $\trinl \bullet \trinr:=|\bullet|_{H^{m}(\Omega)}$ is the energy norm
 and  $\trinl \bullet \trinr_{\text{pw}}:=|\bullet|_{H^{m}(\cT)}:=\| D^m_\text{pw}\bullet\|$ 
 is its piecewise version with the piecewise partial derivatives  $D_\text{pw}^m$ of order $m\in  {\mathbb N}$. 
For non-integer real ${\boldmath{t}}=\ell +s$ with $\ell \in {\mathbb N}_0$ and $0 < s <1$,  the 2D Sobolev-Slobodeckii semi-norm  \cite{Grisvard}
of $f \in H^t(\Omega)$ reads
\begin{align} \label{eq:sobslobo}
|f|_{H^{\boldmath{t}}(\Omega)} := \left( 
\sum_{|\beta|=\ell} \int_\Omega  \int_\Omega 
\frac{|\partial^\beta f(x) - \partial^\beta f(y)|^2 }{|x-y|^{2 +2 s} } \dx \dy   \right)^{1/2}.
\end{align}
The vector space of piecewise polynomials of at most degree $k \in {\mathbb N}$ is
denoted by $P_k(\T)$, $S^k(\T ) := P_k(\T ) \cap C^0({\Omega})  \subset H^1(\Omega)$, and 
$S^k_0(\T ) := P_k(\T ) \cap C^0({\Omega})  \subset H^1_0(\Omega)$ includes the homogeneous boundary conditions.
For a Lebesgue-measurable set $\omega \subseteq {\mathbb R}^n$  and a Lebesgue-measurable function~$v: \omega \rightarrow X$~with values in a finite-dimensional real Banach space $X$, the integral with respect to
the $d$-dimensional Lebesgue measure is denoted by 
$\fint_\omega v \dx := 1/ |\omega| \int_\omega v \dx$. The integral over a $(n-1)$-dimensional hypersurface $\Gamma$ with respect to the
 $(n-1)$-dimensional Hausdorff measure reads as $\int_{\Gamma} v \ds.$
%
 The notation $A \lesssim B$ abbreviates $A \leq CB$ for some positive generic constant $C$, 
which depends on the  solution $u$;
 $A\approx B$ abbreviates $A\lesssim B \lesssim A$.  The set of all $2 \times 2$ real symmetric matrices is denoted by ${\mathbb S}$. The symmetric part of $A \in {\mathbb R}^{2 \times 2}$ is denoted by $\text{sym } (A)$. That is, $\text{sym }(A)= 1/2(A + A^T)$. 
 For any differentiable scalar function $v$  resp. vector field $\Phi=(\phi_1,\phi_2)^T$, 
 $$ {\rm Curl}\, v=\begin{pmatrix}
-\frac{\partial v}{\partial x_2}\\
\frac{\partial v}{\partial x_1}
\end{pmatrix}\quad\text{resp.}
\quad 
 {\rm Curl} \: \Phi:=  {\rm Curl} \begin{pmatrix}
\phi_1\\
\phi_2
\end{pmatrix}
=\begin{pmatrix}
-\frac{\partial \phi_1}{\partial x_2} &\frac{\partial \phi_1}{\partial x_1}\\
-\frac{\partial \phi_2}{\partial x_2} &\frac{\partial \phi_2}{\partial x_1}
\end{pmatrix}.
$$ The symbol $D^1v:=\nabla v$ (resp. $\nabla_\pw v$) denotes the gradient (resp. piecewise gradient) of a scalar function $v$ and $\boldsymbol \sigma:=D^2 u$ 
(resp. $D^2_\pw u $) denotes the Hessian (piecewise Hessian) of  $v$ with the four entries $\sigma_{11}$, $\sigma_{12}$, $\sigma_{21}$, and $\sigma_{22}$ in the form of a $2 \times 2$ matrix.

\section{Nonconforming schemes}\label{section:quasiorthogonality}
\subsection{Continuous model problem} \label{subsection:modelproblem}
Suppose $u \in V:= H^m_0(\Omega)$ solves the $m$-harmonic equation $(-1)^m \Delta^m u=F$ for a given right-hand side $F \in V^*\equiv H^{-m}(\Omega)$
in a  planar bounded Lipschitz domain $\Omega\subset\mathbb{R}^2 $ with polygonal boundary $\partial \Omega$, 
$m \in {\mathbb N}$. The   weak form of this equation reads 
\begin{align} \label{eq:weakabstract}
a(u,v)= F(v) \quad\text{for all }v\in V
\end{align}
with $
a(v,w):= \int_\Omega D^m v:D^m w \dx $ for all $v,w \in V$. The examples in this paper are merely for $m=1,2$, but the abstract framework can be written 
for any $m \in {\mathbb N}$. 
It is well known that \eqref{eq:weakabstract} has a unique solution $u$ and elliptic regularity  \cite{agmon2010,BlumRannacher,gilbargtrudinger2001,necas67} holds in the sense that 
$F \in H^{-s}(\Omega)$ implies $u \in H^{2m-s} (\Omega)$ 
for all $m-\sigma \le s \le m$ with the positive index of elliptic regularity $\sigma$. 
The latter is well established for $m=1,2$ and an overall hypothesis in this paper establishes for other $m$.
 The lowest-order nonconforming finite element schemes suggest linear convergence rates in the energy norm only $\trinl \bullet \trinr:= a(\bullet,\bullet)^{1/2} \approx \|\bullet \|_{V}$ and so  it is sufficient to assume  throughout this paper that the 
quantity $\sigma \le 1$ is the minimum of the index of elliptic regularity and one, whence $0<\sigma\le 1$ depends exclusively on $\Omega$ and $m$. 
The  index of elliptic regularity $\sigma=1$ follows for  a convex polygonal domain $\Omega$ while  $1/2 <\sigma <1$ characterizes a 
nonconvex polygonal domain $\Omega$ \cite{agmon2010,BlumRannacher,gilbargtrudinger2001,necas67}.

\subsection{Nonconforming discretisation}\label{subsection:ncfem}
Suppose that $V_\nc \subset P_m(\T)$ is a nonconforming FE space based on a shape-regular triangulation $\T$ of $\Omega$ into triangles. Let the semi-scalar product $a_\pw$ be defined by the piecewise differential operator  $D^m_\pw$  (the gradient for $m=1$ and the Hessian for $m=2$) 
\[
a_\pw(v_\pw,w_\pw):= \sum_{T \in \T}  \int_T D^m v_\pw : D^m w_\pw \dx \quad\text{for all }v_\pw, w_\pw \in H^m(\T),
\] 
that induces a piecewise $H^m$ seminorm 
$\trinl \bullet \trinr_\pw= a_\pw(\bullet,\bullet)^{1/2}$ that is also a norm in $V_\nc$.  Then $(V_\nc, a_\pw)$ is a (finite-dimensional) Hilbert space so that, given any $F_h \in V_\nc^*$, there exists a unique discrete solution $u_\nc \in V_\nc$ to 
\begin{align} \label{eq:discrete}
a_\pw(u_\nc,v_\nc)= F_h(v_\nc) \quad\text{for all }v_\nc \in V_\nc.
\end{align}

The nonconforming FEM is accompanied with an interpolation operator $I_\nc: V \rightarrow V_\nc$ and a best-approximation property in $P_m(\T)$, i.e.,
\begin{align} \label{eq:best_approx}
a_{\pw}(v-I_\nc v, w_m) & =0 \quad\text{for all } v \in V \text{ and all } w_m \in P_m(\T).
\end{align}
Suppose that there exists a universal constant $\kappa_m >0$ with 
\begin{align} \label{eq:interpolation_estimate}
\|h_\T^{-m} (v-I_\nc v)\| & \le \kappa_m \trinl v- I_\nc v\trinr_\pw \quad\text{for all } v \in V 
\end{align}
for the $L^2$ norm $\|\bullet \|:= \|\bullet \|_{L^2(\Omega)}$ and the piecewise $H^m$ seminorm 
$\trinl \bullet \trinr_\pw$ 
 for the piecewise constant mesh-size $h_\T \in P_0(\T)$ with $h_\T|_T := h_T= \text{ diam }(T) $. 

\begin{rem}[Pythagoras]  The semi-inner product $a_\pw$ in $H^m(\T)
:= \{ v_\pw\in L^2(\Omega): \forall T\in\T, \;   v_\pw|_T\in H^m(T)\}$ 
allows for the concept of  orthogonality:
Two piecewise Sobolev functions  $v_\pw,w_\pw \in H^m(\T)$ are orthogonal if \(a_\pw( v_\pw,w_\pw)=0\) holds and then the Pythagoras theorem 
\[
 \trinl v_\pw+w_\pw  \trinr_\pw^2 =  \trinl v_\pw  \trinr_\pw^2+ \trinl w_\pw \trinr_\pw^2
\]
follows. 
 An example is the orthogonality of $v-I_\nc v$ and  $ w_\nc$ from \eqref{eq:best_approx} that guarantees 
in particular 
\begin{align} \label{eq:Pythogoras}
 \trinl v-w_\nc  \trinr_\pw^2 =  \trinl v -I_\nc v \trinr_\pw^2+ \trinl w_\nc- I_\nc v \trinr_\pw^2
\end{align}
for all $v\in  V $ and $w_\nc\in V_\nc$.
\end{rem}

\begin{rem}[$V+ V_\nc$ is a Hilbert space with inner product $a_\pw$] 
Recall the assumption that $\trinl \bullet \trinr_\pw$ is a norm in $V_\nc$ and that 
some $I_\nc: V \rightarrow V_\nc$  satisfies \eqref{eq:best_approx}-\eqref{eq:interpolation_estimate}. Then $\trinl \bullet \trinr_\pw$ is a norm in $V+ V_\nc$.
Since $\trinl \bullet \trinr_\pw$ is always a seminorm, the definiteness has to be clarified:  Suppose that  $\trinl v+v_\nc \trinr_\pw =0$
 for  some $v \in V$ and $v_\nc \in V_\nc$. This and $w_\nc:=-v_\nc $ imply in \eqref{eq:Pythogoras} that
 $ \trinl v-I_\nc v \trinr_\pw=0=\trinl v_\nc+I_\nc v \trinr_\pw$. Since   $\trinl \bullet \trinr_\pw$ is a norm in $V_\nc$, it follows $v_\nc =-I_\nc v$.
 This shows in  \eqref{eq:interpolation_estimate} that  $v+v_\nc =0$ a.e. as claimed. 
 
 Consequently, 
 $V+ V_\nc$ is a pre-Hilbert space with inner product $a_\pw$ with a complete subspace $V$ and a finite-dimensional subspace $V_\nc$. 
 The completeness of  $V+ V_\nc$  follows from this by standard arguments with the closed subspace 
 $W:=V\cap V_\nc$  and its orthogonal complement $W^\perp$ in $V+ V_\nc$: Let $Q$  and 
 $R$ denote the orthogonal projection onto $W$ and $V\cap W^\perp$, respectively. Given any $v+v_\nc\in W^\perp$
 for $v\in V$ and $v_\nc\in V_\nc$,   the consequence   $v+v_\nc= (1-Q)v+(1-Q)v_\nc$ leads to $W^\perp= (V\cap W^\perp)  \oplus (V_\nc\cap W^\perp)$. 
 The orthogonal split  $v+v_\nc= (1-Q)v+ R(1-Q)v_\nc 
+ (1-R)(1-Q)v_\nc$ leads to  $W^\perp= (V\cap W^\perp) \operp (1-R)(V_\nc\cap W^\perp)$. The space $V+ V_\nc$ is 
therefore the orthogonal sum of the complete subspaces $W$, 
$V\cap W^\perp$, and $(1-R)(V_\nc\cap W^\perp)$ and so complete. 
\end{rem}

\begin{rem}[uniqueness and extension of $I_\nc: V+V_\nc \rightarrow V_\nc$]
Since $(V+ V_\nc, a_\pw)$ is a Hilbert space,  \eqref{eq:best_approx} shows that 
$I_\nc: V \rightarrow V_\nc$  maps any $v$ to its  (unique) best-approximation $I_\nc v$ in $V_\nc$. In other words, the operator $I_\nc$ in 
\eqref{eq:best_approx}-\eqref{eq:interpolation_estimate} is unique and equal to the restriction to $V$ of the best-approximation   $I_\nc: V+V_\nc \rightarrow V_\nc$
that is the orthogonal projection onto $V_\nc$ in the Hilbert space  $(V+ V_\nc, a_\pw)$. Throughout this paper, the symbol  $I_\nc$ denotes this orthogonal projection 
in the extended space   $I_\nc: V+V_\nc \rightarrow V_\nc$ as well as its restriction to arguments in $V$. Notice that  $I_\nc=\text{ id }$ in $V_\nc$ and 
$I_\nc(v+w_\nc) = I_\nc v + w_\nc$ for all $v \in V$ 
and $w_\nc \in V_\nc$.
\end{rem}

\subsection{Examples for $m=1$ and $m=2$}
For $m=1$ (resp. $m=2$), the nonconforming discretization models the Crouzeix-Raviart (resp. Morley) FEM.
Besides the best approximation property \eqref{eq:best_approx}, these examples allow for the $L^2$ orthogonality
\begin{align}\label{eq:L2ortho}
D^m(v-I_\nc v) \perp P_0(\T; {\mathbb R}^{2m}).
\end{align}

\subsubsection{Crouzeix-Raviart FEM for the harmonic equation}\label{sec:cr}
{
The set of vertices $\V$ (resp. edges $\E$) in the shape-regular  triangulation $\T$ of the bounded  polygonal Lipschitz domain  $\Omega$ into triangles is divided into the interior
vertices $\V(\Omega)$ (resp. edges $\E(\Omega)$) and the vertices $\V(\partial\Omega)$
 (resp. edges $\E(\partial\Omega)$) on the boundary $\partial\Omega$. 
Let  $\Pi_k$ denote  the $L^2$ projection onto the space of piecewise polynomials 
${P}_k(\T)$ of total degree at most $k\in {\mathbb N}_0$.   
The mesh-size $h_\cT\in P_0(\cT)$ is defined by
$h_\cT|_T:=h_T:= |T|^{1/2} \approx \text{\rm diam}(T)$ in any triangle 
$T\in\cT$ of area $|T|$;  $|E|$ denotes the length of an edge $E\in\E$.

The CR finite element space  $V_\nc:=\text{CR}^1_0(\T)$ for $m=1$ reads \cite{CrouzeixRaviart}
\begin{eqnarray*}
&{\text{CR}}^1(\T)&:=\big{\{} v_\CR \in P_1(\T) \; {{\big  |}}
\; v_\CR \text{ is continuous at all midpoints of interior edges of } \T  
\big{\}}, \\
&{\text{CR}}_0^1(\T)&:=\big{\{} v_\CR \in \text{ CR}^1(\T) \; {{\big |}}
\; v_\CR \text{ vanishes at all midpoints of boundary edges of } \T  \big{\}}.
\end{eqnarray*}
The CR interpolation  operator $I_\nc:=I_\CR: V \rightarrow  \CR^1_0(\T)$ is defined by 
\[(I_\CR v) (\text{mid }(E)) = \fint_E v \ds \text{ for all } E \in  {\E} \; \text{ for all } v \in V.   \]
It satisfies  \eqref{eq:best_approx}-\eqref{eq:Pythogoras} and \eqref{eq:L2ortho} with 
$\kappa_1 = \sqrt{{j_{1,1}}^{-2}+1/48}$ for the first positive root $j_{1,1}$ of the Bessel function of the first kind \cite{CarstensenGedicke2014,ccdg_2014}, 
that does not even depend on the shape of the triangles. 

\noindent 
For a simply-connected domain $\Omega$,  the  {\em  discrete Helmholtz decomposition} \cite{mindlin} shows
\[P_0(\T;{\mathbb R}^2) =\nabla_\pw \CR^1_0(\T) \operp {\rm Curl } \: (S^1(\T) /{\mathbb R})\]
and, as $h \rightarrow 0$, leads to the Helmholtz decomposition 
\[L^2(\Omega; {\mathbb R}^2) = \nabla H^1_0(\Omega) \operp {\rm Curl }\: (H^1(\Omega)/{\mathbb R}). \]

\subsubsection{Morley FEM for the biharmonic equation} \label{sec:Morley}
The  nonconforming {Morley} finite element space $V_\nc=\M(\T)$ for $m=2$ reads \cite{Morley68}
\begin{eqnarray*}
{\M}'(\T)&:=&\Bigg{\{} v_\M\in P_2(\T){{\Bigg |}}
\begin{aligned}
&\; v_\M \text{ is continuous at the vertices and its normal derivatives } \\
& \nu_E\cdot \nabla_{\NC}{ v_\M} \text{ are continuous at the midpoints of interior edges}
\end{aligned}\Bigg{\}}, 
\\
{\M}(\T)&:=&\Bigg{\{} v_\M\in M'(\T){{\Bigg |}}
\begin{aligned}
&\; v_\M \text{  vanishes at the vertices of }\partial \Omega \text{ and its normal derivatives} \\
&\; \nu_E\cdot \nabla { v_\M} \text{  vanish at the midpoints of boundary edges}
\end{aligned}\Bigg{\}}.
\end{eqnarray*}
For any $v\in V$, the Morley interpolation operator
$I_\nc v:=I_\M v\in \M(\cT)$  is defined by the degrees of freedom 
$$
(I_\M v)(z)=v(z) \text{ for any } z\in \V(\Omega) \text{ and } 
\fint_E\frac{\partial I_\M v}{\partial \nu_E}\ds=\fint_E\frac{\partial v}{\partial \nu_E}\ds \text{ for any } E\in \cE.
$$
The Morley interpolation operator $I_\M : V \rightarrow \M(\T)$ 
satisfies \eqref{eq:best_approx}-\eqref{eq:Pythogoras}  and \eqref{eq:L2ortho} \cite{HuShi09,CCDGJH14}
with  $\kappa_2= 0.25745784465$   \cite{ccdg_2014}   that does not even depend on the shape of the triangles. 

\noindent
For a simply-connected domain $\Omega$, the {\em discrete Helmholtz decomposition} \cite{CCDGJH14} reads 
\[P_0(\T;{\mathbb S}) =D^2_\pw \M(\T) \operp {\rm sym \: Curl } \: (S^1(\T)^2 /{\mathbb R}^2)\]
and, as $h \rightarrow 0$, leads to  the Helmholtz decomposition 
\[L^2(\Omega; {\mathbb R}^2) = D^2 H^2_0(\Omega) \operp {\rm sym \: Curl }\: (H^1(\Omega;{\mathbb R}^2)/{\mathbb R}^2). \]

\subsection{Companion operator}\label{sec:companion}
The interpolation operator $I_\nc : V \rightarrow V_\nc$ allows for a right-inverse $J$ with benefits \cite{CCP}. 
Recall the conforming finite element subspace $V_{\rm c}$ of $V$ and suppose that a linear operator  $J: V_\nc \rightarrow V $ 
and a constant $\Lambda_\jc$ exist with 
\begin{align} 
& I_\nc J = \text{id in } V_\nc, \text{ i.e.,  } I_\nc J v_\nc= v_\nc ,   \label{eq:right_inverse}\\
& \trinl v_\nc- J v_\nc \trinr_\pw \le \Lambda_\jc \trinl v_\nc- v \trinr_\pw, \label{eq:companion_estimate}\\
& v_\nc- J v_\nc \perp P_m(\T) \quad \text{ in } L^2(\Omega) \label{eq:companion_orthogonality}
\end{align}
for all $v_\nc \in V_\nc $  and for all  $v \in V$.

The  companion operator $J$  in this paper (a) allows for an elegant proof (in Subsection~\ref{subsectionAlternativeproof}), 
(b) provides an appropriate right-hand side $F_h:= F(J \bullet) \in V^*_\nc$ in \eqref{eq:weakabstract} for general $F \in V^*$ with
best-approximation (in Subsection~\ref{subsectionBest-approximationforthemodifiedscheme}), and 
(c) controls the inconsistency term $\min_{v\in V}  \trinl  u_\nc -v \trinr_\pw$
in the a posteriori error control (in Section~\ref{sec:aposteriori}).

\begin{example}[companion operators]
 For any $v_\CR \in CR^1_0(\T)$, there exists 
   $J v_\CR \in (S^1_0(\T )+P_4(\T)) \cap V  $  \cite{ccdgmsMathComp} that satisfies 
 \eqref{eq:right_inverse}-\eqref{eq:companion_orthogonality}. 
For any $v_\M \in \M(\T)$, there exists a  $J v_\M \in (HCT(\T )+P_8(\T)) \cap V $ 
  (with the HCT finite element space $HCT(\T) $)
  that  satisfies \eqref{eq:right_inverse}-\eqref{eq:companion_orthogonality} \cite{CCP,Gallistl2014,DG_Morley_Eigen}. 
 \end{example}
 
 \begin{rem}[extra orthogonality in \eqref{eq:companion_orthogonality}] 
The  $L^2$ orthogonality  in \eqref{eq:companion_orthogonality} allows a direct proof of Theorem~\ref{thm:sec2.energynorm0} 
that circumvents the a posteriori error analysis of the consistency 
term as part of the medius analysis \cite{Gudi10}.  
It also allows control over dual norm estimates of the form $\| v_\M- J v_\M \|_{H^{-s}(\Omega)} \lesssim \| h_\T^s ( v_\M- J v_\M )\|$ for $0\le s \le 2$. 
This is critical~e.g.~in eigenvalue analysis or for problems with low-order terms \cite{CCP, Carstensen2019}.
The proof of the best-approximation of Theorem~\ref{thm:energy_norm} for the modified scheme \eqref{eq:modified_discrete}, however, does not require 
the $L^2$ orthogonality  in \eqref{eq:companion_orthogonality}.
\end{rem}

\section{Best-Approximation}\label{section:best_approximation}
The first subsection is a reformulation of the medius analysis in the seminal work \cite{Gudi10}. 

\subsection{First error analysis for $F\in L^2(\Omega)$}\label{subsectionFirsterroranalysisfor} 
Suppose $F\in L^2(\Omega)$ and let $u$ (resp. $u_\nc$)  denote the exact (resp. discrete) solution to \eqref{eq:weakabstract}  (resp. \eqref{eq:discrete}).
The analysis is first illustrated for data $F \equiv f \in L^2(\Omega)$ such that 
$F_h(v_\nc)=F(v_\nc) := \int_{\Omega} f v_\nc \dx$ for all $v_\nc \in V_\nc$ in \eqref{eq:discrete}.

\medskip \noindent
The error $e:=u-u_\nc \in V+ V_\nc$ and its interpolation  $e_\nc:=I_\nc u - u_\nc \in V_\nc$ are naively split into 
\begin{align} \label{eq:split}
\trinl e  \trinr^2_\pw & = a(u, u - J u_\nc) - a_\pw(u_\nc,e_\nc) + a_\pw(u-I_\nc u, J u_\nc - u_\nc).
\end{align}
(The identity  \eqref{eq:split}   follows from elementary algebra in the Hilbert space $V+ V_\nc$ and 
$a_\pw(u_\nc, u-I_\nc u) =0 = a_\pw(I_\nc u, J u_\nc- u_\nc)$ from \eqref{eq:best_approx} and \eqref{eq:right_inverse}.) 
The first and second term on the right-hand side of \eqref{eq:split} allow for \eqref{eq:weakabstract} and \eqref{eq:discrete}, respectively, and their sum is equal to 
$F(u-J u_\nc - e_\nc)$    (recall $F_h \equiv F$). This and the last term in \eqref{eq:split} allow Cauchy inequalities. 
This, \eqref{eq:right_inverse}, and   \eqref{eq:companion_estimate} lead in \eqref{eq:split} to
\begin{align}\label{eq:substitute}
\trinl e \trinr^2_\pw & \le \|h_\T^m f\| \|h_\T^{-m} (1- I_\nc)(u- J u_\nc)\| + \Lambda_\jc
\trinl u - I_\nc u \trinr_\pw  \trinl e \trinr_\pw. 
\end{align}
The interpolation estimate in \eqref{eq:interpolation_estimate} is followed by a triangle inequality, \eqref{eq:companion_estimate}  
with $v=u \in V$ and \eqref{eq:Pythogoras} to derive 
\[
\|h_\T^{-m} (1- I_\nc)(u- J u_\nc)\| \le \kappa_m \trinl u - I_\nc u\trinr_\pw +\kappa_m \Lambda_\jc 
\trinl e \trinr_\pw \le \kappa_m (1+ \Lambda_\jc) \trinl e \trinr_\pw.
\]
The combination with \eqref{eq:substitute} shows best-approximation up to a data approximation term 
\begin{align}\label{eq:gudi}
\trinl u -  u_\nc \trinr_\pw  \le \max\{  \Lambda_\jc, \kappa_m (1+ \Lambda_\jc) \} \left( \trinl u - I_\nc u\trinr_\pw + \|h_\T^m f\|
 \right).
\end{align}
Some local efficiency is enfolded to verify  the {a~posteriori} error estimate $\|h_\T^m f\| \lesssim \osc_m(f,\T)
+  \| u - I_\nc u\|_{H^m(\T)}$ in the examples of \cite{Gudi10}. This is a standard argument in residual-based a posteriori error control and links \eqref{eq:gudi} to the alternate version in \eqref{eq:improved}. 

\subsection{Alternative proof 
for $F\in L^2(\Omega)$}\label{subsectionAlternativeproof} 
The medius analysis in \cite{Gudi10}  does {\it not} employ \eqref{eq:companion_orthogonality} to derive the slightly improved version
\begin{align} \label{eq:improved}
\trinl u -  u_\nc \trinr_\pw  \le 
\constant{} 
\left( \trinl u - I_\nc u\trinr_\pw + \osc_m(f,\T)
 \right)
\end{align}
in terms of the data-oscillations 
\begin{align} \label{eq:osc}
\osc_m(f,\T):= \|h_\T^m (f-\Pi_m f)\|
\end{align}
for $f \in L^2(\Omega)$ and its $L^2$ projection $\Pi_m f$ onto $P_m(\T)$. 
The beneficial property \eqref{eq:companion_orthogonality} allows for a direct proof and this is detailed next.

\noindent 
The key observation is that the split \eqref{eq:split} leads to $F((1-I_\nc)(u-J u_\nc))$ and the contribution 
 $F(u-I_\nc u)$ (that stems from $a(u,u) - a_\pw(u_\nc, I_\nc u)$ by \eqref{eq:weakabstract}-\eqref{eq:discrete}) 
  does {\it not } allow for \eqref{eq:companion_orthogonality}. 
The alternative split
\begin{equation}\label{eq:split_alternative}
\trinl e \trinr^2_\pw  =  a(u, J e_\nc) -a_\pw(u_\nc, e_\nc)  + a_\pw(u, e-J e_\nc)
\end{equation}
(this follows with the arguments in \eqref{eq:split}) does allow for \eqref{eq:companion_orthogonality}. 
Because of \eqref{eq:weakabstract}-\eqref{eq:discrete}, the sum of the first two terms on the right-hand side of \eqref{eq:split_alternative} is equal to 
\[
 F((J-1)e_\nc) = \int_\Omega (f- \Pi_m f)  ((J-1)e_\nc) \dx  \le \osc_m(f,\T) \| h_\T^{-m} (1-J) e_\nc\|
\]
with \eqref{eq:companion_orthogonality}
followed by  a weighted Cauchy inequality in $L^2(\Omega)$ in the end. 
Since $(1-J) e_\nc = (1-I_\nc) (1-J)e_\nc$ from \eqref{eq:right_inverse}, \eqref{eq:interpolation_estimate} 
implies
\begin{align} \label{eq:oscterm}
 \kappa_m^{-1} \| h_\T^{-m} (1-J) e_\nc\| \le \trinl  (1-J) e_\nc \trinr_\pw    \le \Lambda_\jc \trinl e_\nc \trinr_\pw \le \Lambda_\jc \trinl  e \trinr_\pw
\end{align}
with $v=0$ in  \eqref{eq:companion_estimate} for the second last and the Pythagoras theorem (e.g. \eqref{eq:Pythogoras}  with $v=u$ and $w_\nc=u_\nc$) 
 in the last step.

The last term of  \eqref{eq:split_alternative} is recast with $I_\nc(e-Je_\nc )=0$ from \eqref{eq:right_inverse} and
$a_\pw(I_\nc u, e-Je_\nc) =0$ from \eqref{eq:best_approx} as
\begin{align}\label{eqccsec2lastidentityinproofthm:sec2.energynorm0}
 a_\pw(u-I_\nc u, e-Je_\nc) &=  \trinl u-I_\nc u \trinr^2_\pw + a_\pw(u-I_\nc u, (1- J)e_\nc) \nonumber  \\
 & \le  (1+\Lambda_\jc) \trinl u-I_\nc u \trinr_{\pw} \trinl e \trinr_{\pw} 
\end{align}
with a Cauchy inequality and $\trinl  (1-J) e_\nc \trinr_\pw \le \Lambda_\jc \trinl  e \trinr_\pw$ from \eqref{eq:oscterm}  and \eqref{eq:Pythogoras} in the last step. 
The summary of the inequalities  in  \eqref{eq:split_alternative} concludes the  proof of Theorem~\ref{thm:sec2.energynorm0} (with suboptimal constants
as outlined in Remark~\ref{remarkccendofsection2nolast} below).

\begin{thm}[best-approximation up to data approximation] \label{thm:sec2.energynorm0}
The assumptions \eqref{eq:weakabstract}-\eqref{eq:companion_orthogonality} imply \eqref{eq:improved} for $\constant [1]:=\max\{ \kappa_m \Lambda_\jc, 1+ \Lambda_\jc\}.$
\end{thm}

\subsection{Best-approximation for the modified scheme}\label{subsectionBest-approximationforthemodifiedscheme}
The  discrete scheme \eqref{eq:discrete} requires a discrete right-hand side $F_h$  for a general $F \in H^{-m}(\Omega)$. 
The evaluation of $F_h:= F \circ J$ is feasible with $F_h(v_\nc):= F(J v_\nc)$ for all $v_\nc \in V_\nc$ and  
the {\it (modified) nonconforming scheme} seeks the solution $u_\nc \in V_\nc$ to
\begin{equation}\label{eq:modified_discrete}
a_\pw(u_\nc, v_\nc)= F(J v_\nc) \quad\text{for all }v_\nc \in V_\nc.
\end{equation}
%
Let $\Lambda_0$ denote the norm of $1-J$ of the right-inverse $J\in L(V_\nc;V)$ regarded as a linear map between $V_\nc$ and $V$,
\begin{equation}\label{eq:lambda0}
\Lambda_0:=\sup_{v_\nc\in V_\nc\setminus\{0\}} \trinl v_\nc -J v_\nc  \trinr_{\pw} /\trinl v_\nc \trinr_{\pw} \le \Lambda_\jc.
\end{equation}

\begin{thm}[best-approximation] \label{thm:energy_norm}
\noindent (a) Under the assumptions \eqref{eq:best_approx}-\eqref{eq:interpolation_estimate}, \eqref{eq:right_inverse}-\eqref{eq:companion_estimate}, 
the exact solution $u$ to \eqref{eq:weakabstract} and the 
discrete solution $u_\nc$ to \eqref{eq:modified_discrete} satisfy
$$\trinl u -  u_\nc \trinr_\pw  \le \sqrt{1+\Lambda_0^2}\trinl u - I_\nc u\trinr_\pw.$$

\medskip
\noindent (b) The best-approximation constant $\sqrt{1+\Lambda_0^2}=C_{\rm qo}$ is attained in the sense that there exist some $F=a(u,\bullet) \in V^* $ and $u \in V $ with 
$\trinl  u -  u_\nc \trinr_\pw= C_{\rm qo} \trinl  u -  I_\nc u \trinr_\pw >0$ for the discrete solution $u_\nc \in V_\nc$ to \eqref{eq:modified_discrete}.
\end{thm}

\noindent {\it Proof of (a).} 
Recall  the  error $e:=u-u_\nc \in V+V_\nc$ and its  interpolation $e_\nc:= I_\nc u- u_\nc = I_\nc e \in V_\nc$. 
The key identity 
\begin{align}\label{eq:key_identity}
a(u, J e_\nc) = F(Je_\nc) = a_\pw(u_\nc,e_\nc)
\end{align} 
follows directly from \eqref{eq:weakabstract} (resp. \eqref{eq:modified_discrete}) with the test function $J e_\nc \in V $ (resp. $e_\nc$). 
This  proves in \eqref{eq:split_alternative} that 
\begin{equation}\label{eq:modified_discretefollowupcc1}
\trinl e \trinr_\pw^2= a_\pw(u, e-J e_\nc)=a_\pw(u-I_\nc u, e-J e_\nc)
\end{equation}
with $a_\pw(I_\nc u, e-J e_\nc)=0$ in the last step.
The term $a_\pw(u-I_\nc u, e-J e_\nc)$ is estimated in \eqref{eqccsec2lastidentityinproofthm:sec2.energynorm0}.
It is true that $u_\nc$ denotes another discrete solution in  \eqref{eq:oscterm}-\eqref{eqccsec2lastidentityinproofthm:sec2.energynorm0}, but the related arguments remain valid 
verbatim in this proof as well and lead to the upper bound $(1+\Lambda_\jc) \trinl e \trinr_\pw \trinl u- I_\nc u \trinr_\pw$. This concludes  an easy proof of the  
best-approximation result 
 \[
 \trinl u -  u_\nc \trinr_\pw  \le (1+\Lambda_\jc) \trinl u - I_\nc u\trinr_\pw.
 \]
The proof of the optimal best-approximation constant in Theorem~\ref{thm:energy_norm} 
 revisits  \eqref{eq:modified_discretefollowupcc1} 
 and the  Pythagoras theorem $  \trinl e \trinr^2_\pw\  = \trinl u-I_\nc u \trinr_\pw^2+ \trinl e_\nc \trinr_\pw^2$ (e.g. \eqref{eq:Pythogoras}  with $v=e$ and $w_\nc=0$).
The combination with    \eqref{eq:modified_discretefollowupcc1} proves
\[
\trinl e_\nc \trinr_\pw^2=a_\pw(u-I_\nc u, e_\nc -J e_\nc).
\]
The Cauchy inequality and the continuity constant $\Lambda_0$ show  that
\[
\trinl e_\nc \trinr_\pw\le \Lambda_0\trinl   u-I_\nc u\trinr_\pw
\]
(after a division of $\trinl e_\nc \trinr_\pw$, if positive).
This and the above Pythagoras theorem result in  
\[
\trinl e \trinr^2_\pw\  = \trinl u-I_\nc u \trinr_\pw^2+ \trinl e_\nc \trinr_\pw^2\le (1+ \Lambda_0^2) \trinl u-I_\nc u \trinr_\pw^2 .
\]
This concludes the proof under  conditions weaker than   \eqref{eq:companion_estimate} and with a better constant.

\noindent {\it Proof of (b).}  The second assertion on the attainment is immediate for $\Lambda_0=0$; for {$V_\nc \subset V$}; see Remark \ref{rem:cc:section2:||J||>1}  for a discussion of this case. Hence suppose $\Lambda_0>0$ in the sequel.

\medskip
 \noindent Note that $\Lambda_0^2$ from \eqref{eq:lambda0} is an eigenvalue and there exists $v_\nc \in V_\nc$ with $\trinl v_\nc\trinr_\pw=1$ and
\begin{equation} \label{eigen1}
a_\pw(v_\nc-Jv_\nc, w_\nc- Jw_\nc) = \Lambda_0^2 a_\pw(v_\nc, w_\nc) \text{ for all } w_\nc \in V_\nc.
\end{equation}
For $w_\nc=v_\nc$, \eqref{eigen1} shows $\trinl v_\nc- J v_\nc \trinr_\pw=\Lambda_0.$
Since $a_\pw(v_\nc- Jv_\nc, w_\nc)=0$, \eqref{eigen1} implies
\begin{equation} \label{eigen2}
a_\pw(Jv_\nc, Jw_\nc) =(\Lambda_0^2+1)a_\pw(v_\nc, w_\nc)  \text{ for all } w_\nc \in V_\nc.
\end{equation}
Define  $F:=-a(J v_\nc,\bullet)$; so that \eqref{eq:weakabstract} has the exact solution $u=-Jv_\nc \in V$ and the discrete problem for  $u_\nc \in V_\nc$ reads
\begin{equation} \label{eigen3}
a_\pw(u_\nc, w_\nc) =F(Jw_\nc)= -a(Jv_\nc, Jw_\nc) \text{ for all } w_\nc \in V_\nc.
\end{equation}
The combination of \eqref{eigen2}-\eqref{eigen3} shows
$u_\nc=-(1+\Lambda_0^2) v_\nc.$
Consequently, $u-u_\nc=v_\nc-Jv_\nc +\Lambda_0^2 v_\nc$ and $v_\nc=-I_\nc u$. This and the orthogonality  $a(v_\nc - Jv_\nc, v_\nc)=0$ lead to 
\begin{align*}
\trinl u-u_\nc \trinr^2_\pw & =\trinl v_\nc-Jv_\nc \trinr^2_\pw +\Lambda_0^4  \trinl v_\nc \trinr_\pw^2 = 
\Lambda_0^2(1+\Lambda_0^2) >0
\end{align*}
with $\trinl v_\nc-J v_\nc  \trinr_\pw=\Lambda_0$ and $\trinl v_\nc \trinr_\pw=1$ in the last step. Since 
$\trinl u-I_\nc u \trinr_\pw=\trinl v_\nc-J v_\nc  \trinr_\pw=\Lambda_0 $ from the design of $u=-Jv_\nc$ with $I_\nc u=-v_\nc$, this proves
$\trinl u- u_\nc \trinr_\pw= \sqrt{1+\Lambda^2_0} \trinl u-I_\nc u \trinr_\pw.$
\qed

\begin{rem}[$V_\nc \nsubset V$] \label{rem:cc:section2:||J||>1} The case  $\Lambda_0=0$ corresponds to the  optimal case $\|J\|=1$ and implies 
that $J$ is the identity in $V_\nc\subset V$.  This is reserved for the conforming Ritz-Galerkin scheme and excluded throughout this paper.
\end{rem}

\begin{rem}[$V_\nc \neq P_m(\T)$] It is clear 
that $V_\nc \subset P_m(\T)$ and $V_\nc =P_m(\T)$ is excluded in the following sense. If a linear operator $I_\nc: V \rightarrow V_\nc$ exists with  \eqref{eq:best_approx} for $V_\nc = P_m(\T)$ and dimension $n=2$, $m=1$ or $m=2$, then there exists no right-inverse $J$ of $I_\nc$. The proof argues by contradiction, so let $J: P_m(\T) \rightarrow V$ be a right-inverse of $I_\nc$ and start with the observation that \eqref{eq:best_approx} implies $D_\pw^m I_\nc v = \Pi_0 D^m v $ for all $v= J v_\nc \in V$, $v_\nc \in  P_m(\T)$. Since $I_\nc v= v_\nc$ (from $I_\nc J v_\nc =v_\nc$), this means 
\begin{align}\label{eq:ort}
D_\pw^m v_\nc = \Pi_0 {D^m} J v_\nc  \text{ for all }v_\nc \in P_m(\T).
\end{align}
For a contradiction, let $\varphi_{\rm c} \in S^1_0(\T) \setminus P_0(\Omega)$ (so that $\trinl \varphi_{\rm c}\trinr \neq 0$) for $m=1$ and let $\Psi_{\rm c} \in  S^1(\T; {\mathbb R}^2)\setminus P_0(\Omega; {\mathbb R}^2)$ (so that $\trinl \Psi_{\rm c}\trinr \neq 0$)  and then define $v_m \in P_m(\T)$ for $x \in T \in \T$ by 
\begin{align*}
v_1(x)&:= (\text{Curl } \varphi_{\rm c})\cdot (x- \text{ mid }(T) ) \\ \text{ resp. } 
v_2(x)&:=\frac{1}{2} (x- \text{ mid }(T))\cdot \text{ (sym}   (\text{Curl } \Psi_{\rm c})) (x- \text{ mid }(T) ).
\end{align*}
The point is that $\nabla_\pw v_1= \text{Curl } \varphi_{\rm c} \in L^2(\Omega; {\mathbb R}^2)$ resp.  
$D^2_\pw v_2=\text{ sym }  (\text{Curl } \Psi_{\rm c}) \in L^2(\Omega; {\mathbb S})$ is 
$L^2$ perpendicular to $\nabla v$ resp. $D^2v$ for $v:=Jv_m \in V$. Hence
\[
0= \int_\Omega D^m_\pw v_m: D^m v \dx=  \int_\Omega D^m_\pw v_m: D^m Jv_m \dx=
\trinl v_m \trinr_\pw^2
\]
with \eqref{eq:ort} for $v_m \in P_m(\T)$ in the last step. This is a contradiction to the nontrivial choice of $\varphi_{\rm c}$ resp. $\Psi_{\rm c}$ with $\trinl v_m \trinr_\pw>0$.
\end{rem}

\subsection{Interpretation as a Petrov-Galerkin scheme}
This section continues with another interpretation of the modified  nonconforming finite element method as a Petrov-Galerkin scheme in the 
Hilbert space  $(\widehat V, a_\pw)$ for $\widehat V = V+V_\nc$. 
The best-approximation operator $I_\nc: \widehat V\to V_\nc$ in the Hilbert space  $( \widehat V, a_\pw)$ is characterized by the orthogonality $a_\pw(u_\nc, v-I_\nc v)=0$ for 
all $v\in V$.  Since $J$ is a right-inverse,  this and  $v=Jv_\nc$ for $v_\nc\in V_\nc$ imply  $a_\pw(u_\nc, J v_\nc -v_\nc )=0$. Suppose 
 $u\in V$ defines  $F:=a(u,\bullet)\in V^*$ (or vice versa owing to the Riesz isomorphism),  this and  $a_\pw( u_\nc,v_\nc)=a_\pw( u_\nc,J v_\nc)$ lead in 
\eqref{eq:modified_discrete} to
\begin{equation}\label{eqccsection1.2}
a_\pw( u- u_\nc,J v_\nc)= 0\quad\text{for all } v_\nc\in V_\nc
\end{equation}
(and this is equivalent to \eqref{eq:modified_discrete} for  $F=a(u,\bullet)$). This defines a {\it Petrov-Galerkin operator} $ u\mapsto u_\nc$ that is naturally extended to 
$ \widehat V := V+V_\nc$
with \eqref{eqccsection1.2} as follows. Any $\widehat u\in \widehat V$ solves \eqref{eq:modified_discrete} with  $F(Jv_\nc)$ 
replaced by $ a_\pw( \widehat u , J v_\nc)$
for all $v_\nc\in V_\nc$ with the solution  $P\widehat u:= u_\nc\in V_\nc$. Therefore   $P:  \widehat V\to  \widehat V$ 
is characterized by 
\begin{equation}\label{eqccsection1.3}
a_\pw( \widehat v- P  \widehat v ,J v_\nc)= 0\quad\text{for all } v_\nc\in V_\nc \text{ and for all }\widehat v\in \widehat V.
\end{equation}
This characterizes a  Petrov-Galerkin scheme in the Hilbert space  $( \widehat V, a_\pw)$ with the trial space $V_\nc$ and the test space $V_{\rm c}:=JV_\nc$. 
The  Petrov-Galerkin operator $P\in L(\widehat V)$ is regarded as a map from  $\widehat V$ into   $\widehat V$. It is well defined, linear, bounded and 
idempotent. (In particular,  any $u\in V$ is approximated  by $u_\nc=Pu$.)

Based on the interpretation as  a Petrov-Galerkin scheme, several conclusions may be drawn in the subsequent subsection.

\subsection{Comments on the best-approximation constant $C_{\rm qo}$}
The subsection gives certain characterizations of the best-approximation constant $ C_{\rm qo}$ in \eqref{eq:section2withcomments.cc1}
and shows in particular, that the multiplicative constant $C_{\rm qo}= \sqrt{1+\Lambda_0^2}$ in    Theorem~\ref{thm:energy_norm}  is optimal 
 in the resulting estimate 
 \begin{equation}\label{eq:section2withcomments.cc1}
\trinl u -  u_\nc \trinr_\pw  \le C_{\rm qo}  \trinl u - I_\nc u\trinr_\pw
\end{equation}
for the exact solution $u\in V$ to \eqref{eq:weakabstract} and the discrete solution $u_\nc\in V_\nc$ to \eqref{eq:modified_discrete} and an arbitrary right-hand side $F\in V^*$.
To be more precise, adopt the notation of a Petrov-Galerkin operator $P:V\to V_\nc$ from 
\eqref{eqccsection1.3} and specify  
\begin{equation}\label{eq:modified_discretefollowupccbestapproxconst1}
C_{\rm qo} := \sup_{v\in V \setminus V_\nc } \trinl v-Pv \trinr_\pw/  \trinl v-I_\nc v \trinr_\pw
\end{equation}
with a supremum over all $v\in V$ with   $v\ne I_\nc v$ (for this leads to an undefined quotient $0/0$ that can be excluded)  in the definition of the best 
multiplicative constant  \eqref{eq:modified_discretefollowupccbestapproxconst1} for  \eqref{eq:section2withcomments.cc1}.  
To avoid a pathological situation, suppose throughout this paper that 
$ V \setminus V_\nc$ is  non-empty and that $V_\nc\ne\{0\}$ is nontrivial (as well as $V$). (The overall intention is that   the nonconforming 
finite element space $V_\nc\nsubset V$ is {\it not} conforming, but the case  $V_\nc\subset V$ is actually included in the analysis of this section.)
Let $\|J\|:=\|J\|_{L(V_\nc;V)}$ abbreviate  the operator norm of the right-inverse $J: V_\nc\to V$.

\begin{thm}[characterization of best-approximation constant]\label{thm:bestapxconstantsareallequal}
 $C_{\rm qo}= \sqrt{1+\Lambda_0^2}=\|J\|$. 
\end{thm}

\begin{proof} 
{\it Step~1 recalls   $\| 1- P \|_{L(\widehat V)}=C_{\rm qo}$. } This is well known in the mathematical foundations of Petrov-Galerkin schemes and the short proof is 
given for completeness.
Any $\widehat v\in \widehat V=V+V_\nc$ is of the form $\widehat v=v+v_\nc=w+w_\nc$ for
$w:=(1-I_\nc)v\in (1-I_\nc)V=V_\nc^\perp $ and 
$w_\nc =I_\nc v+v_\nc\in V_\nc$.
Since $P\widehat v=Pw+w_\nc$ and $\widehat v-P\widehat v=w-Pw$, the orthogonal split $\widehat v=w+w_\nc$ shows  
\[
\| 1- P \|_{L(\widehat V)}=\sup_{w,w_\nc}  \trinl w-Pw  \trinr_\pw   ( \trinl w  \trinr_\pw^2+\trinl w_\nc  \trinr_\pw^2)^{-1/2}=
\sup_{w\in V_\nc^\perp\setminus\{0\}}  \trinl w-P w \trinr_\pw   / \trinl w  \trinr_\pw
\]
with the abbreviation  $\sup_{w,w_\nc}$ for the  supremum over all  $(w,w_\nc )\in (V_\nc^\perp\times V_\nc)\setminus\{0\}$. This and 
the re-substitution $w=v-I_\nc v$ and  $w-Pw=v-P v $ for $v\in V$ in the second equality conclude the proof,
\[
\| 1- P \|_{L(\widehat V)}=\| 1- P \|_{L(V_\nc^\perp;\widehat V)}=\sup_{v\in V\setminus V_\nc }  \trinl v-P v \trinr_\pw   / \trinl v-I_\nc v  \trinr_\pw=C_{\rm qo} .\qed
\]
{\it Step~2 establishes $\| P \|_{L(\widehat V) } = \| P \|_{L(V_{\rm c};V_\nc) }$} for the conforming companion finite element space $V_{\rm c}:= J V_\nc$.
Since $ \| P \|_{L(V_{\rm c};V_\nc)}\le \| P \|_{L(\widehat V) }$ is obvious (for $ V_{\rm c}\subset \widehat V$),  the proof of the converse  inequality  concerns an arbitrary
$\widehat v \in  \widehat V$ and its orthogonal projection $v_{\rm c}$ onto the (finite-dimensional, whence closed) subspace  $V_{\rm c}:= J V_\nc$ in the Hilbert space 
$(\widehat V, a_\pw)$ with orthogonality written as $\perp$. In particular, $ v_{\rm c}\in V_{\rm c}$ and $\widehat v- v_{\rm c}\perp V_{\rm c}$. Recall from \eqref{eqccsection1.3}
that $ P  \widehat v$ is characterized by $ P  \widehat v\in V_\nc$ and  $\widehat v- P  \widehat v\perp V_{\rm c}$. The two orthogonalities lead to 
$v_{\rm c}-P\widehat v\perp V_{\rm c}$, whence 
$ P  \widehat v=P v_{\rm c}$.
This and  $\trinl v_{\rm c} \trinr_\pw\le \trinl \widehat v \trinr_\pw$ imply
\[
\trinl  P \widehat v \trinr_\pw=\trinl  P v_{\rm c} \trinr_\pw \le \| P \|_{L(V_{\rm c};V_\nc)} \trinl  v_{\rm c} \trinr_\pw\le \| P \|_{L(V_{\rm c};V_\nc)} \trinl \widehat v \trinr_\pw.  \qed
\]
{\it Step~3 computes $ \| P \|_{L(V_{\rm c};V_\nc) }=\| J \| $} with linear algebra. Given any basis $(\psi_1,\dots, \psi_N)$ of $V_\nc$, 
 $(J\psi_1,\dots, J\psi_N)$ forms a basis of $V_{\rm c}$ 
 (for  $v_\nc\in V_\nc$ with  $J v_\nc=0 $ implies $v_\nc= I_\nc J  v_\nc=0$). 
The stiffness matrices $A, B\in \mathbb{R}^{N\times N}$ are defined   by
\[
A_{jk}:= a_\pw(\psi_j,\psi_k)\quad\text{amd}\quad B_{jk}:= a_\pw(J\psi_j,J \psi_k)\quad\text{for all } j,k=1,\dots, N.
\]
Then $A$ and $B$ are SPD and characterize the Petrov-Galerin map $P:V_{\rm c}\to V_\nc$ in coefficient vectors $x,y\in \mathbb{R}^N$ for 
$v_\nc=\sum_{j=1}^N x_j \psi_j$ and $v_{\rm c}=\sum_{j=1}^N y_j J \psi_j$ by  $v_\nc=P v_{\rm c}$ is equivalent to $Ax=By$. 
The norm  $\| P \|_{L(V_{\rm c};V_\nc)}$   of $P:V_{\rm c}\to V_\nc$ involves the norms with the squares 
$ \trinl  P v_{\rm c} \trinr_\pw^2=\trinl  v_\nc \trinr_\pw^2= x\cdot Ax$ and $\trinl  v_{\rm c} \trinr_\pw^2= y\cdot By$ and lead, for $x,y\in  \mathbb{R}^N\setminus\{0\}$ with  $Ax=By$,  to
$ y\cdot By=x \cdot A B^{-1}Ax $. This and the Rayleigh-Ritz $\min$-$\max$ principle  show 
\[
\| P \|_{L(V_{\rm c};V_\nc)}^2=\sup_{v_{\rm c} \in V_{\rm c} \setminus\{0\}}    \frac{      \trinl  P v_{\rm c} \trinr_\pw^2}{   \trinl  v_{\rm c} \trinr_\pw^2}=
\sup_{x\in  \mathbb{R}^N\setminus\{0\}}    \frac{ x\cdot Ax }{ x \cdot A B^{-1}Ax } = \lambda 
\]
for the maximal eigenvalue of an eigenpair $(\lambda, x)\in \mathbb{R}_+\times  \mathbb{R}^N\setminus\{0\}$ of the generalized
eigenvalue problem $Ax=\lambda  A B^{-1}Ax$. Since the latter is equivalent to $ Bx=\lambda A x$, the maximal eigenvalue is
\[
\lambda=\sup_{x\in  \mathbb{R}^N\setminus\{0\}}     \frac{  x \cdot Bx }{x\cdot Ax}  
=\sup_{v_\nc \in V_\nc \setminus\{0\}}     \frac{    \trinl  J v_\nc \trinr_\pw^2 }{  \trinl  v_\nc \trinr_\pw^2}
= \|J\|^{2}.  \qed
\]
{\it Step~4 verifies $\| J \|^2=1+\Lambda_0^2 $.}   Since $J$ is a right-inverse to $I_\nc$,  the Pythagoras theorem 
$ \trinl  J v_\nc \trinr_\pw^2=\trinl  v_\nc - J v_\nc \trinr_\pw^2+  \trinl  v_\nc \trinr_\pw^2$ concludes the proof by the definition of $\|J\|$ and $\Lambda_0$:
\[
\| J \|^2=\sup_{v_\nc \in V_\nc \setminus\{0\}}     \frac{  \trinl  v_\nc - J v_\nc \trinr_\pw^2+  \trinl  v_\nc \trinr_\pw^2   }{  \trinl  v_\nc \trinr_\pw^2}
=1+\sup_{v_\nc \in V_\nc \setminus\{0\}}  \frac{  \trinl  v_\nc - J v_\nc \trinr_\pw^2  }{  \trinl  v_\nc \trinr_\pw^2}=1+\Lambda_0^2.
\qed
\]
{\it Step~5 concludes the proof with Kato's lemma.} Any oblique projection (that is a bounded linear map $P\in L(\widehat V)$ in a Hilbert space $\widehat V$ with $P^2=P$)
satisfies $\|1-P\|_{L(\widehat V)} =\|P\|_{L(\widehat V)}$ provided $0\ne P\ne 1$. This is known as Kato's lemma \cite{kato60} with many very different proofs in \cite{Szy06}.
The summary of all steps concludes the proof.
\end{proof}

\begin{rem}
Veeser et al. \cite{veeser_zanotti2} provide the bound $C_{\rm qo} {=} \|J\|$,  while  Theorem~\ref{thm:energy_norm} 
gives  $C_{\rm qo}=\sqrt{1+\Lambda_0^2}$ and 
both results are optimal. 
\end{rem}

\begin{rem}[$J$ vs $V_{\rm c}$] The Petrov-Galerkin scheme and its operator $P$ require the trial space $V_\nc$ and the test space $V_{\rm c}$, but no knowledge 
about $J$. However, $V_{\rm c}:= J(V_\nc)$ depends very much on the right-inverse. In fact, given any $V_{\rm c}\subset V$ of dimension $\dim V_{\rm c}=\dim V_\nc\in \mathbb{N}$,
there is at most one right inverse:   If $I_\nc$ is not injective as a linear map from $V_{\rm c}$ into $V_\nc$, or equivalently  (for $\dim V_{\rm c}=\dim V_\nc$) it is not surjective, then 
there exists no right-inverse $J:V_\nc\to V$ with $V_{\rm c}=J(V_\nc)$. But otherwise $J$ exists and is uniquely determined by $V_{\rm c}$.  
(If $J: V_\nc\to V_{\rm c}$ is a linear map and $I_\nc J$ is identity in $V_\nc$ then $J$ is injective and hence bijective for  $\dim V_{\rm c}=\dim V_\nc$.
Then $J$ is the inverse of $I_\nc$, when   $I_\nc$ is regarded  as a bijection of $V_{\rm c}$ onto $V_\nc$.) 
\end{rem}

\begin{rem}[$\|J\|>1$]\label{rem:cc:section2a:||J||>1}
Theorem~\ref{thm:bestapxconstantsareallequal} asserts $\|J \|\ge1$ and the optimal case $\|J\|=1$ 
  for the conforming Ritz-Galerkin scheme is discussed in Remark~\ref{rem:cc:section2:||J||>1}. For the truly nonconforming situation $V_\nc\nsubset V$,
the quality of the scheme is controlled by $C_{\rm qo}=\|J\|>1$. 
\end{rem}
 
\begin{rem}[another characterization of $\Lambda_0$] 
Theorem~\ref{thm:bestapxconstantsareallequal} also reveals  the  identity 
\(
\Lambda_0=\| P\|_{L(V_\nc^\perp;\widehat V)}
\)
the orthogonal complement of $V_\nc$ in $\widehat V$.  This identity is behind the refined proof of Theorem~\ref{thm:energy_norm}. 
The proof of it considers $Pw\perp w\in V_\nc^\perp$ and  the Pythagoras theorem  
$\trinl w-Pw  \trinr_\pw ^2 = \trinl Pw  \trinr_\pw^2+ \trinl w \trinr_\pw^2$ for the first equation in 
\[
\| 1-P\|_{L(V_\nc^\perp;\widehat V)}^2  =\sup_{w\in V_\nc^\perp\setminus\{0\}} ( \trinl Pw  \trinr_\pw ^2 + \trinl w \trinr_\pw^2) \trinl w \trinr_\pw^{-2}=1+\| P\|_{L(V_\nc^\perp;\widehat V)}^2
 \]
Theorem~\ref{thm:bestapxconstantsareallequal} and the last line in Step~1 of its proof show that this is equal to $1+\Lambda_0^2$. \qed
\end{rem}

\begin{rem}[post-processing]
The error $u-J u_\nc$ of the post-processed discrete solution can be measured in all Sobolev scales $H^{s}(\Omega)$ for $-m\le s\le m$ and we start with
$s=m$ and the affirmative estimate
\[
\left| \trinl u- u_\nc  \trinr_\pw -  \trinl  u-J u_\nc\trinr\right| \le \Lambda_{\rm J} \trinl u- u_\nc  \trinr_\pw\le \Lambda_{\rm J}  C_{\rm qo}  \trinl u- I_\nc  u \trinr_\pw.
\]
(This follows from a reverse triangle inequality and \eqref{eq:companion_estimate} plus Theorem \ref{thm:energy_norm}.) This estimate compares with 
errors in the continuous and piecewise energy norms up to a best-approximation term on the right-hand side. In particular, the post-processing is quasi-optimal.
However, it is {\it not} quasi-optimal in comparison with the best-approximation in $V_{\rm c}:= J V_\nc$: 
Otherwise, with the 
Petrov-Galerkin operator $P:V\to V_\nc$ from \eqref{eqccsection1.3}, any inequality 
 $ \trinl v- JP v  \trinr_\pw\lesssim  \trinl v-   v_{\rm c} \trinr_\pw$ for all   $v\in V$ and $v_{\rm c}\in V_{\rm c}$ implies,
 in particular, that   $v_{\rm c}=J P v_{\rm c}$ for  all $v=v_{\rm c}\in V_{\rm c}$.
Since $J$ is a right-inverse, the last identity shows that $I_\nc v_{\rm c}=P v_{\rm c}$ for all $v_{\rm c}\in V_{\rm c}$. 
But this implies  $\| P \|_{L(V_{\rm c};V_\nc) }\le 1$.  Step 3 in the proof of Theorem \ref{thm:bestapxconstantsareallequal} leads to $\|J\|=1$. This is
characterized in Remark \ref{rem:cc:section2:||J||>1} as the conforming case $V_\nc=V_{\rm c}\subset V$.
\end{rem}

\begin{rem}[sharper constants in Theorem~\ref{thm:sec2.energynorm0}] \label{remarkccendofsection2nolast}
The constant $\constant [1] $
in \eqref{eq:improved} is not much larger than  $\sqrt{1+\Lambda_0^2 } $.
Follow the proof until \eqref{eq:oscterm} and replace $ \Lambda_\jc$ by  $\Lambda_0$ and ignore the last estimate. 
Utilize $\Lambda_0$ for  $\trinl  (1-J) e_\nc \trinr_\pw \le \Lambda_0 \trinl  e_\nc  \trinr_\pw$ in the equality 
\eqref{eqccsec2lastidentityinproofthm:sec2.energynorm0}. This leads to an estimate  for the right-hand side of \eqref{eq:split_alternative} 
that includes the additive term  $\trinl u-I_\nc u \trinr_{\pw}^2$ (still from \eqref{eqccsec2lastidentityinproofthm:sec2.energynorm0}) that 
arises also in the Pythagoras theorem $ \trinl   e \trinr_\pw^2 = \trinl   e_\nc \trinr_\pw^2 +\trinl u-I_\nc u \trinr_{\pw}^2$ on the left-hand side.
This proves (after a division of $\trinl e_\nc \trinr_\pw$ if positive)
\[
 \trinl   e_\nc \trinr_\pw    \le  \Lambda_0 (\kappa_m  \osc_m(f,\T)  +  \trinl u-I_\nc u \trinr_{\pw}) .
\]
This identity and the  Pythagoras theorem again conclude the proof of
\[
\trinl e_\nc \trinr_\pw^2 \le   (1+\Lambda^2_0)\trinl u-I_\nc u \trinr_{\pw}^2 
+ 2 \Lambda^2_0  \kappa_m \osc_m(f,\T) \trinl u-I_\nc u \trinr_{\pw}+ \kappa^2_m  \osc_m(f,\T)^2     . 
\]
The right-hand side is equal to $(\trinl u-I_\nc u \trinr_{\pw}, \osc_m(f,\T) ) \cdot  A  (\trinl u-I_\nc u \trinr_{\pw}, \osc_m(f,\T) ) $ with the
symmetric $2\times 2$ matrix $A$ that has entries $A_{11}=1+\Lambda^2_0$, $A_{12}= \Lambda^2_0  \kappa_m$, and $A_{22}= \kappa^2_m$.
The eigenvalues of  $A$ can be computed
and the bigger  one leads to a bound \eqref{eq:improved} with 
\[
2(1+\Lambda_0^2)  <2 \constant [1]^2= 1+\kappa_m^2+\Lambda_0^2+ \sqrt{(1-\kappa_m^2+\Lambda_0^2)^2+4\kappa_m^2\Lambda_0^4} .
\]
Since $\kappa_m$ is small, the constant $\constant [1]$ 
is not significantly larger than $C_{\rm qo}$ in \eqref{eq:section2withcomments.cc1}. \qed
\end{rem}

\section{Weaker and piecewise Sobolev norm estimates}\label{sec:lower_order}
\subsection{Lower-order error estimates}
This section lists further conditions sufficient for a lower-order a priori error estimate for the 
discrete problem \eqref{eq:discrete} under the hypotheses \eqref{eq:best_approx}-\eqref{eq:interpolation_estimate} and \eqref{eq:right_inverse}-\eqref{eq:companion_estimate}.

\medskip
The duality argument relies on the elliptic regularity  and the approximation properties of the dual and discrete problem with constants $0 < \sigma \le 1$ and $ 0<   C_{\rm reg}(s), C_{\rm int}(s)  <\infty$ such that
for any $G\in H^{-s}(\Omega)$ with $m-\sigma \le s \le m$, 
the weak solution $ z \in V=H^m_0(\Omega) $  to  the PDE $(-1)^m \Delta^m z=G $  satisfies  $ z \in V \cap H^{2m-s}(\Omega) $
and 
\begin{align}
\|z \|_{ H^{2m-s}(\Omega)}&\le C_{\rm reg}(s)  \| G \|_{H^{-s}(\Omega)} , \label{eq:extrareg} \\
\trinl  z - I_\nc z \trinr _\pw &\le C_{\rm int}(s)  h_{\rm max}^{m-s} \|z \|_{ H^{2m-s}(\Omega)} .  
\label{eq:est}
\end{align}
Since  in general $u_\nc \in V_\nc$ may not belong to { $H^{ m-s }(\Omega)$}, the post-processing $J u_\nc$ is considered in the lower-order a priori error estimate.

\begin{thm}[lower-order error estimates]\label{thm:aux1}
Suppose the hypotheses of Theorem \ref{thm:energy_norm}, adopt the notation \eqref{eq:extrareg}-\eqref{eq:est} for some $s$ with   $m-\sigma \le s \le m$. Given any $F\in H^{-s}(\Omega) $, 
 the solutions $u \in V$ to \eqref{eq:weakabstract} and $u_\nc \in V_\nc$ to \eqref{eq:discrete} satisfy 
$(a) \; \|u - J u_\nc \|_{H^{s}(\Omega)} \le 
\constant{}\label{lower1}(s)  h_{\max}^{m-s}~\trinl u-u_\nc\trinr_\pw$ and
$(b) \;   \|u-u_\nc\|_{{H^s}(\T)} \le   \constant{}\label{lower2}(s) h_{\max}^{ m-s}  
\trinl  {u -  u_\nc} \trinr_\pw.$
\end{thm}}

\noindent{\it Proof of (a).} 
The  duality argument is applied for the exact solution $u\in V $ to  \eqref{eq:weakabstract}
and the post-processing $v:= J u_\nc  \in V$ for the discrete solution  $u_\nc \in V_\nc$ to \eqref{eq:discrete}. 
The duality of  {$H^{s}_0(\Omega)$ and  $H^{-s}(\Omega)$} reveals
\begin{align*}
\|u -v\|_{H^{s}(\Omega)} & = \sup_{0 \ne G \in H^{-s}(\Omega)} \frac{G(u-v)}{\|G\|_{H^{-s}(\Omega)}}= G(u-v)
\end{align*}
and the supremum is attained for some  $G\in H^{-s}(\Omega)\subset V^*$ with norm  $\|G\|_{H^{-s}(\Omega)} =1$ 
(owing to a corollary of the Hahn-Banach theorem). The functional has a 
 unique Riesz representation $z\in V$ in the Hilbert space $(V,a)$,  $a(\bullet,z)=G\in V^*$;
i.e., $z\in V$   the weak solution to the PDE  $(-1)^m \Delta^m z=G $. 
The   elliptic  regularity \eqref{eq:extrareg}
leads to $ z \in V \cap H^{2m-s}(\Omega) $ and \eqref{eq:extrareg}-\eqref{eq:est}; in particular, 
\begin{align} \label{eq:u-va}
\|u -v\|_{H^{s}(\Omega)}& =a(u-v,z)\quad\text{and}\quad
\|z \|_{ H^{2m-s}(\Omega)}\le C_{\rm reg}(s).
\end{align}
Let $
v:= J u_\nc \in V$ and $\zeta:= J I_\nc z \in V$, where $u_\nc$ (resp. $z$) denote the discrete (resp. dual) solution.
The test function $\zeta= J I_\nc z  \in V$ in  \eqref{eq:weakabstract} and the test function $ I_\nc z$ in \eqref{eq:discrete} lead to
\[
a_\pw(u_\nc, I_\nc z) = F(\zeta) = a(u, \zeta).
\]
This and elementary algebra with $a_\pw((1-J) u_\nc , I_\nc z) =0= a_\pw(u_\nc, (1-J)I_\nc z) $  (from \eqref{eq:best_approx} and \eqref{eq:right_inverse}), result in
\begin{align}\label{eq:a(u-v,z)}
a(u-v,z)  = a(u-v,z - \zeta) +a_\pw(u_\nc-v, \zeta-I_\nc z).
\end{align}
From  \eqref{eq:companion_estimate} deduce that
\begin{align} \label{eq:diff_interpolation}
\trinl v- u_\nc \trinr_\pw & \le \Lambda_\jc \trinl u- u_\nc \trinr_\pw  \quad \text{and}\quad
\trinl \zeta- I_\nc z\trinr_\pw  \le \Lambda_\jc \trinl z- I_\nc z \trinr_\pw.
\end{align}
The combination of \eqref{eq:diff_interpolation}    
and triangle inequality leads to 
\begin{align} \label{eq:u-v}
\trinl u-v\trinr & \le \trinl u-u_\nc  \trinr_{\pw} + \trinl v-u_\nc   \trinr_{\pw} \le (1+\Lambda_\jc ) \trinl  u-u_\nc \trinr_\pw.
\end{align}
Similar arguments for  $(z,\zeta, I_\nc z)$ replacing $(u,v,u_\nc)$  reveal
\begin{align} \label{eq:z-zeta}
\trinl z-\zeta\trinr & \le (1+\Lambda_\jc )   \trinl z-I_\nc z\trinr_\pw.
\end{align}
A Cauchy inequality  for the terms on the right-hand side of \eqref{eq:a(u-v,z)}  followed by an application of \eqref{eq:diff_interpolation}-\eqref{eq:z-zeta} implies that  
$a(u-v,z )$  is bounded by 
 $ ((1+\Lambda_\jc)^2 +\Lambda_\jc^2) \trinl u-u_\nc \trinr_\pw  \trinl  z- I_\nc z \trinr_\pw$.
A substitution to \eqref{eq:u-va} leads to 
\[
\|u -v\|_{H^{s}(\Omega)}\le {((1+\Lambda_\jc)^2+ \Lambda_\jc^2)}   \trinl u-u_\nc\trinr_\pw \trinl  z- I_\nc z \trinr_\pw.
\]
This,   \eqref{eq:est}, and (\ref{eq:u-va}b)  lead to  
$\constant[\ref{lower1}](s):=C_{\rm int}(s) C_{\rm reg}(s)  ((1+\Lambda_\jc)^2+\Lambda_\jc^2) $. \qed

\medskip
\noindent{\it Proof of (b).}   The 
Sobolev-Slobodeckii semi-norm over $\Omega$ involves double integrals over $\Omega\times\Omega$ and so is larger   than or equal to the sum of the 
integrals over $T\times T$ for all the triangles $T\in\T$, i.e.,    $ \sum_{T\in\T} |\bullet  |_{H^s(T)}^2\le  |\bullet  |_{H^s(\Omega)}^2$ for any $m-1<s<m$.
Hence  Theorem \ref{thm:aux1}.a implies for all $s\text{ with }m -\sigma \le s \le m$ and a constant $\constant[\ref{lower1}](s)$ that
\begin{equation}\label{eqrefccnewa1b}
 \|u -J u_\nc \|_{H^{s}(\T)} \le \constant[\ref{lower1}](s)\, h_{\max}^{m-s}  \trinl u-u_\nc \trinr_\pw.
\end{equation}

\noindent The equivalence of the Sobolev-Slobodeckii norm and the norm by interpolation of Sobolev spaces \cite[Remark 9.1]{LM72}, for instance for a fixed 
reference triangle   $T=T_{\rm ref}$ with 
$\constant{}\label{ccint3}(s) =\constant[\ref{ccint3}](s,T_{\rm ref})$, provides for  ${w}:=(u_\nc-J u_\nc )|_T\in H^m(T)$  the estimate 
\begin{equation}\label{eqininterpolationofSobolevnormosonTc}
 \|{w} \|_{H^s(T)}   \le \constant[\ref{ccint3}](s)\,   \|{w} \|_{ H^{m-1}(T)}^{m-s}  \|{w} \|_{H^{m}(T)}^{s-m+1} \quad\text{for } m-1< s<m.
 \end{equation}
 A straightforward transformation of Sobolev norms \cite[Theorem 3.1.2]{Ciarlet}  shows
\eqref{eqininterpolationofSobolevnormosonTc} for any triangle  $T\in\T $ with 
$\constant[\ref{ccint3}]  (s)=\constant[\ref{ccint3}]  (s,T) 
= \kappa^{1+s} \constant[\ref{ccint3}](s,T_{\rm ref}) $ for the condition number $\kappa=\sigma_1/\sigma_2 $ of the affine transformation $a+Bx$
of $T_{\rm ref}$ to $T$ with the $2 \times 2 $ matrix $B$ and its positive singular values $\sigma_2 \le \sigma_1$. 
A more detailed analysis \cite{book_nncc} reveals that $\constant[\ref{ccint3}](s)$ exclusively depends on $s$ (but exploits
singularities as $s$ approaches the end-points $m-1$ and $m$). The estimate \eqref{eqininterpolationofSobolevnormosonTc} shows the first inequality in 
\[
\constant[\ref{ccint3}] (s)^{-2}\, 
\|{w} \|_{H^s(T)}^2   \le    \|{w} \|_{ H^{m-1}(T)}^{2(m-s)}  \|{w} \|_{H^{m}(T)}^{2(s-m+1)} 
\le   \|{w} \|_{ H^{m-1}(T)}^{2} +  \|{w} \|_{ H^{m-1}(T)}^{2(m-s)}  |{w} |_{H^{m}(T)}^{2(s-m+1)}
\]
with the subadditivity $(a+b)^p\le a^p + b^p $ for $a,b\ge 0$ and $0<p=s-m+1<1$ (e.g. from the concavity of $x\mapsto x^p$ for non-negative $x$) in the last step.  An elementary estimate is followed by the  Young's inequality $ab\le a^p/p+b^q/q$ for $p=(m-s)^{-1}$, $q=(s-m+1)^{-1}$, $a=  \| h_T^{-1} {w} \|_{ H^{1}(T)}^{2(m-s)}$, and
$b= |{w} |_{H^{2}(T)}^{2(s-m+1)}$ 
to prove
\[
 h_{\max} ^{ 2(s-m)} \|{w} \|_{ H^{m-1}(T)}^{2(m-s)}  |{w}|_{H^{m}(T)}^{2(s-m+1)}\le  \|h_\T^{-1} {w} \|_{ H^{m-1}(T)}^{2(m-s)}|{w} |_{H^{m}(T)}^{2(s-m+1)}\le 
 \|h_\T^{-1}  {w} \|_{ H^{m-1}(T)}^{2}  +     |{w} |_{H^{m}(T)}^{2} .
\]
This and the trivial estimate $h_{\max} ^{ 2(s-m+1)} \le \text{\rm diam}(\Omega)^{2(s-m+1)}$ lead to 
$\constant{}\label{ccint4}(s)=\constant[\ref{ccint3}] (s)^{2} (1+\text{\rm diam}(\Omega)^{2(s-m+1)})$ in 
\[
 \|{w} \|_{H^s(T)}^2 \le \constant[\ref{ccint4}](s) h_{\max} ^{ 2(m-s)} \left( \|h_T^{-1}  {w} \|_{ H^{m-1}(T)}^{2}  +   |{w} |_{H^{m}(T)}^{2} \right).
\]
The sum over all those contributions over $T\in\T$ reads
\begin{align*}
\| u_\nc-J u_\nc \|_{H^s(\T)}&\le  \sqrt{\constant[\ref{ccint4}](s)} h_{\max}^{ m-s}  \left(  \| h_\T^{-1}(u_\nc-J u_\nc) \|_{H^{m-1}(\T)}+ 
\trinl  u_\nc-J  u_\nc \trinr_\pw\right).
\end{align*}
An inverse estimate  shows $\|h_T^{-1} (u_\nc-J u_\nc) \|_{ H^{m-1}(T)} \lesssim  h_T^{-m} \| u_\nc-J  u_\nc \|_{L^2(T)}$ for all $T \in \T$  with a constant in $\lesssim$ that  depends  only on the shape-regularity. This, \eqref{eq:interpolation_estimate}, a triangle inequality, \eqref{eqrefccnewa1b},  and \eqref{eq:companion_estimate}
conclude the proof of  $(b)$.
\qed

\subsection{Precise convergence rates for quasi-uniform meshes}\label{sec:precise}
This subsection comments on the sharpness of the {a~priori} estimates of Theorem~\ref{thm:aux1} and discusses conditions sufficient for the convergence 
rate $ t:= \min\{2 \sigma, m+\sigma -s\} $  of the  error $u- J u_\nc \in V=H^m_0(\Omega)$ in the norm of $H^s(\Omega)$ for $-m \le s \le m$ as 
$h_{\rm max} \rightarrow 0$. Under the 
generic assumption $\trinl u-u_\nc \trinr_{\pw} \approx h_{\max}^{\sigma}$, the estimate
\begin{align} \label{eq:sharp}
\|u - J u_\nc \|_{H^s(\Omega)} \approx h_{\max}^t \text { for }  -m \le s < m+1/2
\end{align}
holds for a class of quasi-uniform shape-regular triangulations $\T$  with sufficiently small maximal mesh-size $h_{\max} \ll 1$.
Figure \ref{fig:rates} depicts the convergence rate $t$  in \eqref{eq:sharp} as $h_{\max} \rightarrow 0$.
\begin{center}
\def\sig{2}     
\def\m{5}       
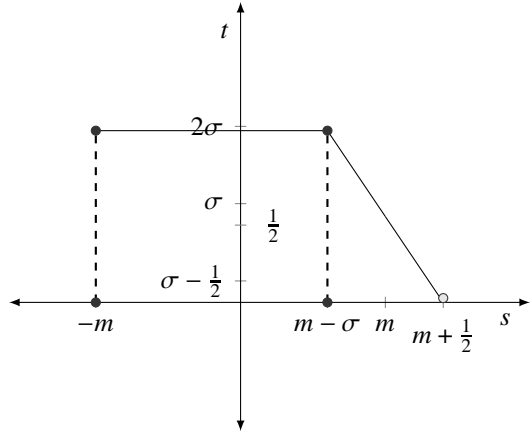
\begin{figure}
\centering
\begin{tikzpicture}[>=latex] 
\begin{axis}[
axis x line=middle,
axis y line=middle,
xmin=-8,
xmax=10,
ymin=-3,
ymax=7,
axis line style={latex-latex},
every axis x label/.style={
	at={(ticklabel* cs:0.92)},
	anchor=west,
},
every axis y label/.style={
	at={(ticklabel* cs:0.89)},
	anchor=south,
},
xtick={-5,3,5,7},
xticklabels={$-m$,$m - \sigma$,$m$, $m+\frac{1}{2}$},
ytick={0.5,2.3,4.1},
yticklabels={$\sigma - \frac{1}{2}$,$\sigma$, $2\sigma$},
extra y ticks={4},
extra y tick labels={$2\sigma$},
extra y tick style={
	yticklabel style={yshift=1ex, xshift=0.5ex,anchor=east}
},
extra y ticks={1.8},
extra y tick labels={$\frac{1}{2}$},
extra y tick style={
	yticklabel style={yshift=-1ex, xshift=3.5ex,anchor=east}
},
xlabel={$s$},
ylabel={${t}$},
xlabel style={below right},
ylabel style={above left},
]
\addplot [mark=none,domain=-5:3] {4};    
\addplot [mark=none,domain=3:5] {7 - x };
\addplot [mark=none,domain=5:7] {7 - x };
\draw[-] [thick,dashed] (axis cs: -5,0) -- (axis cs: -5,4);
\draw[-] [thick,dashed] (axis cs: 3,0) -- (axis cs: 3,4);
 \node at (axis cs:-5,0) [circle, scale=0.3, draw=black!80,fill=black!80] {};
  \node at (axis cs:-5,4) [circle, scale=0.3, draw=black!80,fill=black!80] {};
   \node at (axis cs:3,0) [circle, scale=0.3, draw=black!80,fill=black!80] {};
    \node at (axis cs:3,4) [circle, scale=0.3, draw=black!80,fill=black!80] {};
        \node at (axis cs:7,0.1) [circle, scale=0.3, draw=black!80,fill=black!10] {};
\end{axis}
\end{tikzpicture}
\caption{Rates of convergence $t$ of $\|u-Ju_{\rm nc}\|_{H^s(\Omega)}$ as $h_{\rm max} \rightarrow 0$ for $-m \le s < m+1/2$.} \label{fig:rates}
\end{figure}
\end{center}
\begin{thm}\label{thm:new} 
Given $F\in H_0^{   \max\{ 0, - s \}  } (\Omega) $ 
 for $-m \le s \le m+1/2$, let $u \in V$ solve \eqref{eq:weakabstract} and let $u_\nc \in V_\nc$ solve \eqref{eq:modified_discrete}.  Suppose  $\sigma <m$,  $ t:= \min\{2 \sigma, m+\sigma -s\} $, and that  for a sequence  of quasi-uniform shape-regular triangulations
\begin{equation}\label{eq:starr}
 h_{\max}^{\sigma} \lesssim \trinl u - u_\nc \trinr_\pw 
 \end{equation}
holds. Then $\|u - J u_\nc \|_{H^s(\Omega)} \approx h_{\max}^t$ as $h_{\max}\to 0$.
\end{thm}

The proof of the theorem for $-m \le s \le m-\sigma$ rests on  the subsequent lemma.

\begin{lem} \label{sec:sharp_one} 
For $F\in H_0^{   \max\{ 0, - s \}  } (\Omega) $ 
and $-m \le s \le m- \sigma$, the exact solution $u \in V$ to \eqref{eq:weakabstract} and the discrete solution $u_\nc \in V_\nc$ to \eqref{eq:modified_discrete} satisfy 
\[ \trinl u- u_\nc\trinr_\pw^2  \le 2 \kappa_m h^m_{\max} (1+\Lambda_{\rm J}) \|F\| \trinl u - u_\nc \trinr_\pw + \|F\|_{H^{-s}(\Omega)}   \|u- Ju_\nc\|_{H^s(\Omega)}. \]
\end{lem}

\begin{proof}
The weak solution $u$ to $(-1)^m \Delta^m u =F$ satisfies $a(u,u)=F(u)$. The discrete solution $u_\nc$ to \eqref{eq:modified_discrete} 
satisfies $a_\pw(u_\nc, 2 I_\nc u - u_\nc) = F(2 JI_\nc u - J u_\nc)$. Consequently, 
\[ \trinl u - u_\nc \trinr^2_\pw = F(Ju_\nc - u) + 2F( u - J I_\nc u). \]
The first term on the right-hand side of the above expression  is $\le \|F\|_{H^{-s}(\Omega)} \|u - J u_\nc\|_{H^s(\Omega)}$  while the second term, for $F \in L^2(\Omega) $, is twice
\begin{align*}
F(  u - J I_\nc u) & \le \kappa_m h^m_{\max} \|F\| \trinl u - J I_\nc u \trinr_{\pw} \\
& \le \kappa_m (1+ \Lambda_{\rm J}) h^m_{\max} \|F\|  \trinl u - I_\nc u \trinr_{\pw}
\end{align*}
with {\eqref{eq:interpolation_estimate} in the first step and a triangle inequality plus \eqref{eq:companion_estimate}} in the last step. 
A combination of the estimates concludes the proof.
\end{proof}

\begin{proof}[Proof of Theorem~\ref{thm:new} for $-m \le s \le m- \sigma$] 
Suppose that $\Omega$ is a (possibly nonconvex) polygon with elliptic regularity 
index $1/2 <\sigma \le 1$ and $\T$ is quasi-uniform with mesh-size $h_{\max}$. 
The estimate $\trinl u - u_\nc \trinr_\pw \lesssim h_{\max}^{\sigma}$ follows from Theorem \ref{thm:energy_norm} and the interpolation approximation estimate $\trinl u - I_\nc u \trinr_\pw \lesssim h_{\max}^{\sigma}$ for $u \in H^{m+\sigma}(\Omega)$. 
The additional assumption  \eqref{eq:starr}  is that the convergence rate is not better than $\sigma$. This and Lemma \ref{sec:sharp_one} imply
\[ 
\constant{} \label{con:lemsharp} h_{\max}^{2 \sigma} \le \|u - J u_\nc\|_{H^s(\Omega)} + h_{\max}^{m + \sigma} 
\]
for a positive constant $\constant[\ref{con:lemsharp}] \approx 1$. In case $m=1=\sigma$, this estimate is inconclusive. 
But otherwise $2 \sigma < m + \sigma$ (e.g. for $m= 2$ or $\sigma<1$) and 
\(
h_{\max}^{2 \sigma} (\constant[\ref{con:lemsharp}]  -  h_{\max}^{m -\sigma} ) \le \|u - J u_\nc\|_{H^s(\Omega)} 
\)
shows, at least for sufficiently small $h_{\max} \ll 1$, that  $h_{\max}^{2 \sigma}  \lesssim \|u - J u_\nc\|_{H^s(\Omega)} $. The 
Sobolev embedding for   $ s \le m- \sigma$ shows $\|u - J u_\nc\|_{H^s(\Omega)} \lesssim \|u - J u_\nc\|_{H^{m-\sigma}(\Omega)}$.
Theorem~\ref{thm:aux1}  and $\trinl u - u_\nc \trinr_\pw \lesssim h_{\max}^{\sigma}$ imply 
 $\|u - J u_\nc\|_{H^{m-\sigma}(\Omega)} \lesssim h_{\max}^{2 \sigma} $. The combination of all those estimates concludes the proof for 
 $-m \le s \le m- \sigma$.
\end{proof}

Suppose $\T$ is quasi-uniform and $0 <\varepsilon < 1/2 $   throughout the remaining parts of this subsection.
A quasi-uniform mesh allows inverse estimates even in the Sobolev-Slobodeckii semi-norm.  Throughout the remaining part of this subsection, the generic factors hidden in $\lesssim$ may also depend on $\varepsilon$.
\begin{lem}[inverse estimate] \label{lem:sharp2} 
It holds 
\(
  \varepsilon^{1/2} \|u - J u_\nc \|_{H^{m+\varepsilon}(\Omega)} \lesssim h_{\max}^{\sigma-\varepsilon} |u|_{H^{m+\sigma}(\Omega)}.
\)
\end{lem}

\begin{proof}
The best-approximation property leads to  $\|u- Ju_\nc\|_{H^m(\Omega)}  \lesssim h^{\sigma}_{\max} |u|_{H^{m+\sigma}(\Omega)} $ and controls the 
low-order contributions. Hence the proof focusses on the Sobolev-Slobodeckii semi-norm of each partial-dertivative $\partial^\alpha$ of order $|\alpha|=m$. 
Let $f:= \partial^\alpha u$, $g:= \partial^\alpha Ju_\nc$. The triangle inequality motivates the split
\[ 
|f-g|_{H^{\varepsilon}(\Omega)} \le |f-\Pi_0 f|_{H^{\varepsilon}(\Omega)} + |\Pi_0 f -g|_{H^{\varepsilon}(\Omega)}.
\]
The first term is controlled by a Poincare inequality for the Sobolev-Slobodeckii semi-norm 
\[\|f- \Pi_0 f\|_{H^{\varepsilon}(\Omega)}  \lesssim \varepsilon^{-1/2} h_{\max}^{s-\varepsilon} |f|_{H^s(\Omega)}.\]
(This is a standard estimate for the reference triangle e.g.~by interpolation and a scaling argument completes the proof with a generic constant that may depend on $\varepsilon$. The
claimed multiplicative constant $\approx \varepsilon^{-1/2}$ follows from a refined estimation \cite{book_nncc}.)
The second term is  an inverse estimate for $Ju_\nc \in V$ that 
is a piecewise polynomial (e.g.~with respect to a HCT refinement of $\T$)  with mesh-sizes equivalent to $h_\T$. 
Since  $0 <\varepsilon < 1/2$, any piecewise polynomial $p_r \in P_r(\T)$ of a degree at most  $r \in {\mathbb N}_0$ belongs to
$H^{\varepsilon}(\Omega)$ and the inverse estimate 
$\|p_r \|_{H^{\varepsilon}(\Omega)} \lesssim  \varepsilon^{-1/2} \|h_\T^{-\varepsilon} p_r \|_{L^2(\Omega)}$ holds  \cite{DFGHS} for a quasi-uniform mesh
(the multiplicative constant $\approx \varepsilon^{-1/2}$ follows from a refined estimation \cite{book_nncc}). Since this is available for the above
piecewise polynomial $\Pi_0 f -g$, it implies
 \[
  \varepsilon^{1/2}  |\Pi_0 f -g|_{H^{\varepsilon}(\Omega)} \lesssim \|h_{\T}^{-\varepsilon} (\Pi_0 f -g) \| \lesssim { h_{\rm min}^{-\varepsilon}} \trinl Ju_\nc -I_\nc u \trinr_\pw
  \]
for the minimal mesh-size $h_{\min}:= \min_{T \in \T} h_T$ in $\T$ and $h_{\min} \approx h_{\max}$ for a  quasi-uniform triangulation. A triangle inequality with \eqref{eq:Pythogoras}, \eqref{eq:companion_estimate}, and Theorem \ref{thm:energy_norm}  lead to 
\begin{align*}
\trinl I_\nc u - J u_\nc  \trinr_{\pw} & \le \trinl  I_\nc u - u_\nc \trinr_{\pw} + \trinl  u_\nc - J u_\nc  \trinr \\
&  \le C_{\qo}(1+\Lambda_{\rm J}) \trinl  u - I_\nc u \trinr_{\pw} \lesssim h_{\max}^{\sigma} |u|_{H^{m+\sigma}(\Omega)}.
\end{align*}
The combination of the above estimates for each derivative $f= \partial^{\alpha} u$ of order $m=|\alpha|$ shows
\( \varepsilon^{1/2} \|u- J u_\nc\|_{H^{m+\varepsilon}(\Omega)} \lesssim  h_{\max}^{\sigma-\varepsilon} |u|_{H^{m+\sigma}(\Omega)} \).
This concludes the proof for a quasi-uniform triangulation.
\end{proof}

\begin{lem}[interpolation]\label{lemma interpolation} 
If~$0 < s<m$ and $0 <\varepsilon < 1/2 $, then \\
\mbox{}\hfil \( \| u- J u_\nc \|_{H^m(\Omega)}^{s+\varepsilon} \lesssim h_{\max}^{s(\sigma-\varepsilon)}  |u|_{H^{m+\sigma}(\Omega)}^s \|u- J u_\nc\|^{\varepsilon}_{H^{m-s}(\Omega)}\).
\end{lem}

\begin{proof} 
The interpolation \cite{tartar2010introduction} of $H^m(\Omega)$ between $H^{m-s}(\Omega)$ and
$H^{m+\varepsilon}(\Omega)$ for $e:=u- J u_\nc  \in V \cap H^{m+\varepsilon}(\Omega)$  leads to
$\|e\|_{H^m(\Omega)} \lesssim \|e\|^{s/(s+\varepsilon)}_{H^{m-s}(\Omega)}
\|e\|^{\varepsilon/(s+\varepsilon)}_{H^{m+\varepsilon}(\Omega)}.$
Lemma~\ref{lem:sharp2}  reveals  $\|e\|_{H^{m+\varepsilon}(\Omega)} \lesssim  h_{\max}^{\sigma-\varepsilon} |u|_{H^{m+\sigma}(\Omega)}$
and  concludes the proof.
\end{proof}

\begin{proof}[Proof of Theorem~\ref{thm:new} for $m- \sigma\le s \le m$] 
The upper bound for the error $e:=u- J u_\nc  \in V $ 
 in the norm of  $H^s(\Omega)$ follows from  Theorem~\ref{thm:aux1}.
The assumption   \eqref{eq:starr} and Lemma~\ref{lemma interpolation} reveal 
$h_{\max}^{\sigma+m-s}\lesssim   |u|_{H^{m+\sigma}(\Omega)}^{s/(\sigma+\varepsilon)} \|e\|_{H^{s}(\Omega)}$. 
This concludes the proof for  $m- \sigma\le s \le m$.
\end{proof}

\begin{proof}[Proof of Theorem~\ref{thm:new} for $m\le s < m+1/2$] 
The upper bounds for quasi-uniform meshes follow from Lemma \ref{lem:sharp2}.
The lower bounds are derived by interpolation for  $e:=u- J u_\nc  \in V \cap H^{m+\varepsilon}(\Omega)$.  For the convex coefficients 
$\theta={\varepsilon}/({\sigma+\varepsilon})$ and $1-\theta={\sigma}/({\sigma+\varepsilon})$ in 
$m=\theta(m-\sigma) +(1-\theta)(m+\varepsilon)$, an  interpolation \cite{tartar2010introduction} shows 
\[ 
\trinl e \trinr \le C(\varepsilon)  \| e \|_{H^{m-\sigma}(\Omega)}^{\theta} \|e\|_{H^{m+\varepsilon}(\Omega)}^{1-\theta}
\]
for some constant $C(\varepsilon)$ that exclusively depends on $\varepsilon$ and $\Omega$. This and  Theorem~\ref{thm:aux1}
result in   $\trinl e \trinr^{\sigma+\varepsilon} \lesssim h_{\max}^{2\sigma} \varepsilon\|e\|_{H^{m+\varepsilon}(\Omega)}^{\sigma}$. 
Recall   \eqref{eq:starr} and derive  the asserted lower bound $ h_{\max}^{\sigma- \varepsilon} \lesssim \|e\|_{H^{m+\varepsilon}(\Omega)} $. 
This concludes the proof.
\end{proof}

\section{Comments and examples  of data in $H^{-{m}}(\Omega)$}\label{sec:comments}
\subsection{The space $H^{-m}(\Omega)$} 
\label{sec:characterization}
 The space $H^{-m}(\Omega)$ is the dual to $H^m_0(\Omega)$ and is a real Hilbert space. It is known that for any functional $F \in H^{-m}(\Omega)$,  there exist $f_{\alpha} \in L^2(\Omega)$ for all multi-indices $\alpha$ with $|\alpha| \le m$ such that 
 \begin{equation} \label{eq:formf}
 F(v) = \sum_{|\alpha| \le m} \int_{\Omega} f_{\alpha} \partial_{\alpha} v \dx \text{ for all } v
 \in H^m_0(\Omega)
 \end{equation}
(with multi-index notation $\alpha=(\alpha_1,\alpha_2,\cdots, \alpha_n) \in {\mathbb N}_0^n$ and
 $|\alpha| = \alpha_1+ \alpha_2 + \cdots + \alpha_n$ for the partial derivatives 
 $\partial_{\alpha} v:= \partial^{|\alpha|} v /(\partial^{\alpha_1} x_1 \cdots \partial^{\alpha_n} x_n)$ etc). 
 The family $(f_\alpha: |\alpha| \le m )$ of $L^2(\Omega)$ functions with \eqref{eq:formf} is {\em not} unique. For instance,
 given any ${F} \in L^2(\Omega)$ one may solve  {$(-1)^m \Delta^{m} u ={ F}$} for a weak solution $u \in H^m_0(\Omega)$ and 
 consider $G:=  D^m u \in L^2(\Omega; {\mathbb R}^{(n^m)})$ with 
 \[ F(v) = a(u, v) = \int_\Omega  G: D^m v \dx \text{ for all } v \in H^m_0(\Omega).
 \] 
Hence $
f_\alpha :=
\begin{cases}
F \text{ if } \alpha=(0,0,\cdots,0) \\
0  \text{ if } 1 \le |\alpha| \le m
\end{cases}
$
and 
$
f_\alpha :=
\begin{cases}
0 \text{ if } |\alpha| \le m-1 \\
{  G_\alpha}  \text{ if }  |\alpha| = m
\end{cases}
$ with  $G_\alpha=\partial_\alpha u$  define two sets of data $(f_\alpha: |\alpha| \le m )$ 
on the right-hand side of \eqref{eq:formf} for the same functional $F$. It is also known 
\cite[Lemma 15.9, p.78]{tartar2010introduction} or \cite[Thm.1.~in~5.9, p.253]{Evans} for $m=1$ that $\|F\|_{H^{-m}(\Omega)}$ 
is the infimum of $\sqrt{\sum_{|\alpha | \le m} \|f_\alpha\|^2_{L^2(\Omega)}}$ for all $(f_\alpha: |\alpha| \le m )$ with \eqref{eq:formf}.

\medskip \noindent
It cannot be overemphasized that the preprocessing of the data set $(f_\alpha: |\alpha| \le m )$ in \eqref{eq:formf}  is a subtle issue that essentially depends on the way the right-hand side $F$ is provided. For instance, point forces may apply, a right-hand side $f \in L^1(\Omega)$ for $m \ge 2$ and line integrals may be modeled for $m \ge 1$. The main idea will
be explained for a model example with   the right-hand side $F \in H^{-m}(\Omega)$, $\Omega \subset {\mathbb R}^2$  in \eqref{eq:formf}   of the form 
\begin{equation} \label{F:formata}
F(v):= \int_{\Omega} G: D^m v\dx + \int_{\Omega} gv \dx\quad\text{for all }v\in V=H_0^m(\Omega)
\end{equation}
for given  $G \in L^2(\Omega; {\mathbb R}^{2^m})$ and  $g \in L^2(\Omega)$.
Consider the vector space  $H({\rm div}^m, \Omega; {\mathbb R}^{2^m})$ of all $L^2$ functions $Q \in L^2(\Omega; {\mathbb R}^{2^m})$ 
such that the $m$-th  distributional divergence ${\rm div}^m Q \in L^2(\Omega)$ is a Lebesgue function: 
$ q\in L^2(\Omega)$ is called the $m$-th divergence of some $Q \in L^2(\Omega; {\mathbb R}^{2^m})$ if, for all $v  \in D(\Omega)$,
$\int_\Omega Q: D^m v \dx= (-1)^m \int_\Omega q v \dx.$
This holds for all $v \in H^m_0(\Omega)$ as well (by density of $D(\Omega)$ in $H^m_0(\Omega)$) and one writes $q=(-1)^m {\rm div}^m \: Q$. Hence any $Q$ and any $v \in H^m_0(\Omega)$ satisfy
\begin{equation}\label{modf}
F(v)=   \int_\Omega (G- Q) :D^m v \dx +  \int_\Omega (g+  {\rm div}^m \: Q) \: v \dx 
\end{equation} 
In other words, the data $G$ and $g$ of $F$ in \eqref{F:formata} can be modified to $G':= G- Q$ and
$g':= g+ {\rm div}^m \: Q $  without changing the action of $F$ on the continuous level.
This data preprocessing may reduce the data oscillations as they arise e.g. in the a~posteriori error control. 

\subsection{Point forces $(m=2)$}\label{sec:pointforces} The mechanical  model often involves point loads. The Dirac delta functional $\delta_z \in H^{-2}(\Omega)$ is the evaluation of the continuous test function $u \in V =H^2_0(\Omega) \hookrightarrow C({\overline \Omega})$ at a fixed point $z \in \overline{\Omega}$, called atom of the delta functional.   Unlike for $m=1$, $\delta_z: V \rightarrow {\mathbb R}$ is linear and bounded for $m \ge 2$. 
It is well known that (if $\langle \bullet, \bullet  \rangle$ abbreviates the duality brackets in $H^{-2}(\Omega) \times
H^{2}(\Omega) $ that extends the $L^2$ scalar product)
\[
 \langle  \delta_z,v \rangle =v(z) = \frac{1}{2 \pi} \int_\Omega \Delta v (x) \log |x-z| \dx\quad\text{for all }v\in V=H_0^m(\Omega)
. \]
(cf. e.g. \cite[Ex.1, page 82]{Folland} or \cite[Thm 1, page 23]{Evans} for further details and proofs.) Hence $\delta_z$ is represented by $G{(x)}:=\frac{1}{2 \pi} \log |x-z|  1_{2 \times 2}$ for $x \in \Omega$ with the $2 \times 2$ unit matrix $1_{2 \times 2}$ and $g \equiv 0$. Since $G \in L^2(\Omega; {\mathbb S})$, the results of this survey apply to this case as well, but the practical computation may avoid integration with the weakly singular function $\log |x-z|$ and just apply $\delta_z$ on the point evaluation. 
The point force  at $a \in A \subset \V(\Omega) $ does not contribute to the upper bound below  in Proposition~\ref{prop} when located at an interior  vertex of the triangulation. 
The model example will suppose this for simplicity and because it can easily arranged by appropriate triangulations in practice at least for a small number of point forces. 

Before the next subsection continues with the corresponding model example, the effect of a general location shall at least be illuminated.
 
\begin{example}[point force not located at a vertex] 
Suppose  that there is a point force at $z \in \Omega$ that is not correlated to the triangulation $\T$. The original Morley FEM has to model $\delta_z$ also in case 
$z \in \partial T_+ \cap \partial T_-=E$ is shared by  two triangles $T_+ $ and $T_-$. The test function $v_\M \in \M(\T)$ is (in general) discontinuous at $z$ and 
\begin{align}\label{fh}
F_h&:= \mu \delta_{T_+,z}+ (1-\mu) \delta_{T_-,z} 
\end{align}
applies with the piecewise point evaluation $\langle  \delta_{T_{\pm}, z }, v_\pw \rangle_{H^2(\T)}:= (v_\pw|_{T_{\pm}})(z)$ for all $v_\pw \in H^2(\T)$ 
and a convex coefficient $0 \le \mu \le 1$; $\mu =1/2$ appears  natural. 
The choice of $\mu $ does not affect the modified Morley FEM, but it will do so for the original scheme in general. 
 The a~priori error analysis of the best-approximation for the modified Morley FEM remains valid and the Proposition~\ref{prop} is easily modified. 
 A standard Bramble-Hilbert argument shows for $w \in H^2(T)$, $  T \in \T$, that  $w:= v_\M- Jv_\M $ satisfies
$$
\|w\|_{L^{\infty}(T)} \le  \constant{}\label{conj} h_T |w|_{H^2(T)} 
$$
(In fact, the $P_1$ FEM interpolation error analysis shows in the above displayed inequality for a left-hand side 
$|w|_{H^1(T)}$  based on $w \in H^2(T)$ with $w(z)=0$ for all vertices $z \in \V(T)$.
The $H^1$ seminorm and the  $L^{\infty}$ norm do not scale in 2D  \cite{Ciarlet} and this leads to the assertion.) 
It follows that the absolute value of 
\[F_h(v_\M)-F(Jv_\M)= \mu(w|_{T_+})(z)+ (1-\mu) (w|_{T_-})(z) \]
 is smaller than or equal to 
$ \mu \|w\|_{L^{\infty}(T_+) } + (1-\mu )\|w\|_{L^{\infty}(T_-) }
\le  \constant[\ref{conj}]  \left( \mu h_{T_+} +(1-\mu ) h_{T_-}  \right)  \trinl  { w} \trinr_\pw.$
This leads to the additional data-oscillation term $ \mu h_{T_+} +(1-\mu ) h_{T_-}$. 
\end{example}

\subsection{Model example}
The smoother $J$ in the right-hand side of the discrete problem \eqref{eq:modified_discrete}   leads to an a~priori best-approximation result for  
general data $F \in H^{-m}(\Omega)$. The a~posteriori error control requires data evaluation in explicit residual-based a~posteriori error estimates. 
Throughout this section the particular format of the right-hand side $F \in H^{-m}(\Omega)$ in \eqref{eq:weakabstract}  is assumed of the form 
\begin{equation} \label{F:format}
F(v):= \int_{\Omega} G: D^m v\dx + \int_{\Omega} gv \dx +\sum_{a \in A}  \beta(a) v(a) \text{ for all } v \in V
\end{equation}
for given  $G \in L^2(\Omega; {\mathbb R}^{(n^m)})$, $g \in L^2(\Omega)$, $A$ is a finite set of interior vertices, 
and $a \in A \subset \V(\Omega)$ is the position 
and  $\beta(a)$  the strength of the point force.  
The {\it natural right-hand side} $F_h:=\widehat{F}|_{V_{\nc}} \in V_\nc^*$ in the discrete problem for  \eqref{F:format} reads,
for all $\widehat{v} =v+v_\nc \in \widehat{V} \equiv V+V_\nc$,
\begin{equation} \label{Fhat:format}
\widehat{F}(\widehat{v}):= \int_{\Omega} G: D_\pw^m \widehat{v}\dx + \int_{\Omega} g \widehat{v} \dx  +\sum_{a \in A}  \beta(a) \widehat{v}(a). 
\end{equation}

\subsection{Comparison}\label{sec:comparison} The two schemes \eqref{eq:modified_discrete} and \eqref{eq:discrete} for $F_h = \widehat{F}|_{V_\nc}$ 
from \eqref{Fhat:format} are in fact different. While the modified scheme \eqref{eq:modified_discrete} with right-hand side $F \circ J$ 
enjoys a best-approximation property, the original version for the right-hand side $\widehat{F}|_{V_\nc}$ from \eqref{Fhat:format} does so only up to data-oscillations.

\begin{prop}\label{prop} Given  $F$ and ${\widehat F}$ in \eqref{F:format}-\eqref{Fhat:format}, let $u_\nc^{\rm org} \in V_\nc$ solve \eqref{eq:discrete}
and  let $u_\nc^{\rm mod}$ solve \eqref{eq:modified_discrete}. Then
\( \Lambda_0^{-1} \trinl u_\nc^{\rm org} - u_\nc^{\rm mod}\trinr_\pw \le \|G-\Pi_0 G\| +\kappa_m \osc_m(g,\T).\)
\end{prop}

\begin{proof} Since 
$\displaystyle  \trinl u_\nc^{\rm org} - u_\nc^{\rm mod}\trinr^2_\pw=\widehat{F} ((1-J) (u_\nc^{\rm org} - u_\nc^{\rm mod})) \le \| \widehat{F} \circ (1-J) \|_{V^*_\nc} \: \trinl u_\nc^{\rm org} - u_\nc^{\rm mod}\trinr_\pw$,
it remains to analyse $\| \widehat{F} \circ (1-J) \|_{V^*_\nc}$. For any $v_\nc \in V_\nc$ with $\trinl v_\nc \trinr_\pw=1$, the definition \eqref{Fhat:format} shows
\begin{align*}
\widehat{F} \circ (1-J) (v_\nc) & = \int_\Omega G : D^2_\pw(v_\nc - J v_\nc) \dx +\int_{\Omega} g (v_\nc - J v_\nc ) \dx \\
& = \int_{\Omega} (G - \Pi_0  G): D^2_\pw(v_\nc - J v_\nc) \dx +\int_{\Omega} (g-\Pi_m g)(v_\nc - J v_\nc ) \dx
\end{align*}
with the assumptions that  $\Pi_0 D^2  J v_\nc = D^2_\pw v_\nc$ and $\Pi_m J v_\nc = v_\nc$ for $m=1$ and $m=2$ in the last step that holds for CR and Morley FEM (see Sections \ref{sec:cr} and \ref{sec:Morley}). This and $\|h_\T^{-m} (v_\nc - J v_\nc) \| \le \kappa_m \trinl v_\nc - J v_\nc \trinr_\pw$ from \eqref{eq:interpolation_estimate} and \eqref{eq:right_inverse}  prove
$$  \widehat{F} \circ (1-J) (v_\nc) \le (\| G - \Pi_0  G \| + \kappa_m \: \osc_m(g, \T))\trinl v_\nc - J v_\nc \trinr_\pw. $$
The definition of $\Lambda_0$ in \eqref{eq:lambda0} shows $ \trinl v_\nc - J v_\nc \trinr_\pw \le \Lambda_0 
\trinl v_\nc \trinr_\pw = \Lambda_0.$ The combination of the preceding estimates proves 
$\widehat{F} \circ (1-J) (v_\nc)  \le \Lambda_0  (\| G - \Pi_0  G \| + \kappa_m \: \osc_m(g, \T)).$ This concludes the proof. 
\end{proof}

\noindent The data-oscillation term $\osc_m(g,\T)=o(h_{\rm max})$ can be of higher-order if $g \in H^1(\T)$ is piecewise smooth for $m \ge 2$, 
while  this is less clear for the data-oscillation  $\|G- \Pi_0 G\|$. 

\medskip \noindent The subsequent example shows that $(a)$ the original scheme can be optimal while the modified one is not,  $(b)$  the bound in the Proposition~\ref{prop} 
is sharp in general (at least if $g\equiv 0$). The methodology is similar to the analysis in Theorem~\ref{thm:bestapxconstantsareallequal}; in fact $(c)$ is another 
proof of the optimality of $C_\qo=1+\Lambda_0^2$ (recall  $\Lambda_0>0$ throughout the paper).

\begin{example} 
There exist $G:= D^m J z_\nc \in D^m V_{\rm c}$ and $g:=0$ in \eqref{F:format} for some $z_\nc \in V_\nc$ such that the exact solution $u \in V$ to \eqref{eq:weakabstract}, the discrete solution $u_\nc^{\rm org}$ to the original scheme \eqref{eq:discrete} and the discrete solution
 $u_\nc^{\rm mod} \in V_\nc$ to \eqref{eq:modified_discrete}  with $F$ and $F_h= \widehat{F}|_{V_h}$ from \eqref{F:format}-\eqref{Fhat:format}, satisfy 
  $(a) \; u_\nc^{\rm org} = I_\nc u, $ 
  $(b) \; \trinl u_\nc^{\rm org} - u_\nc^{\rm mod} \trinr_\pw = \Lambda_0^2,$ and 
   $(c) \; \trinl u - u_\nc^{\rm mod} \trinr^2_\pw = \trinl u- u_\nc^{\rm org} \trinr^2_\pw  + \Lambda_0^4 = {  \Lambda_0^2(1+\Lambda_0^2). }$
\end{example}

\medskip \noindent{\it Proof of $(a)$.} Given $V_\nc$ and $J$ from Subsections~\ref{subsection:ncfem} and \ref{sec:companion}, there exists  $z_\nc \in V_\nc$ with $\trinl z_\nc \trinr_\pw=1$ and $\trinl z_\nc - J z_\nc \trinr_\pw = \Lambda_0 \trinl z_\nc \trinr_\pw= \Lambda_0$.
Let $G:= D^m J z_\nc$ and recall $\Pi_0 G= D^m_\pw I_\nc J z_\nc  = D^m_\pw z_\nc$ with $\|\Pi_0 G \|= \trinl z_\nc \trinr_\pw=1$ and $\|G-\Pi_0 G\| = \trinl z_\nc - J z_\nc \trinr_\pw=\Lambda_0$. For all $v_\nc \in V_\nc$, \eqref{eq:best_approx} shows
\[a_\pw(z_\nc, v_\nc) = a_\pw(J z_\nc, v_\nc) = \int_{\Omega}G : D^m_\pw v_\nc  \dx \]
with $G:= D^m J z_\nc$ in the last step.  Hence $u_\nc^{\rm org}=z_\nc$. Since $g=0$, for $u= J z_\nc$, \eqref{F:format} shows
\[ a(u,v)= a(Jz_\nc, v) =\int_\Omega G: D^m v \dx  \text{ for all } v \in V. \]
This and $I_\nc u = I_\nc J z_\nc = z_\nc = u_\nc^{\rm org}$ proves the equality in $(a)$.

\medskip \noindent{\it Proof of $(b)$.}  Recall 
\begin{align*}
\Lambda_0^2&= \trinl z_\nc - J z_\nc \trinr_\pw^2= a_\pw(J z_\nc - z_\nc, J z_\nc) \\
& = a(u,u) - a_\pw(z_\nc,z_\nc) = \int_{\Omega} G : D^m J z_\nc \dx  - \int_{\Omega} G : D^m_\pw z_\nc \dx  \\& = a_\pw( u_\nc^{\rm mod}-u_\nc^{\rm org},z_\nc).
\end{align*}
A Cauchy inequality in the semi-scalar inner product $a_\pw$ with $\trinl z_\nc \trinr_\pw=1$ leads to 
\[ \Lambda_0^2= a_\pw(u_\nc^{\rm mod}-u_\nc^{\rm org},z_\nc) \le 
\trinl u_\nc^{\rm mod}-u_\nc^{\rm org} \trinr_\pw.\] 
For $g=0$ and $\|G-\Pi_0 G\|= \Lambda_0$ from above, the Proposition~\ref{prop} shows
$\trinl u_\nc^{\rm mod}-u_\nc^{\rm org} \trinr_\pw \le \Lambda_0^2$.
This concludes the proof of $(b)$.
\qed

\medskip \noindent{\it Proof of $(c)$.}  
The assertion $(c)$ follows from the above and the Pythagoras theorem 
\[  \trinl u- u_\nc^{\rm mod}\trinr_\pw^2 = \trinl J z_\nc - z_\nc \trinr_\pw^2 + \trinl  u_\nc^{\rm mod}-  u_\nc^{\rm org} \trinr_\pw^2. \qed \]

\subsection{Counterexamples for best-approximation}
The best-approximation property for the original nonconforming FEM with the natural right-hand side fails in general: The data oscillation terms cannot be removed.
Suppose  $\Omega$ a  simply-connected domain throughout this subsection.

\subsubsection{Counterexample to best-approximation of CRFEM for $m=1$}
Suppose $|\T| \ge 2 $ so that there exists some $\beta_\CR \in \CR^1(\T) \setminus S^1(\T)$. Since the quotient space 
$S^1(\T)/ {\mathbb R}$ is a Hilbert space with scalar product $a_\pw(\bullet,\bullet)$, 
the linear functional $v_{\rm c} \mapsto a_\pw(\beta_\CR,v_{\rm c}) \in (S^1(\T) \setminus {\mathbb R})^*$ 
has a Riesz representation $\beta_{\rm c}$ so that $b_\CR:=\beta_\CR - \beta_{\rm c}\in \CR^1(\T) \setminus S^1(\T)$ 
satisfies $a_\pw(b_\CR, v_{\rm c})=0$ for all $v_{\rm c} \in S^1(\T)$. 
Without loss of generality, suppose $\trinl b_\CR \trinr_\pw=1$. 
Since  ${\rm Curl }_\pw b_\CR$ is a rotated gradient $\nabla_\pw b_\CR$ etc. in $2$D,   
the $L^2$ orthogonality ${\rm Curl }_\pw b_\CR \perp {\rm Curl} \: S^1(\T)$ in $L^2(\Omega; {\mathbb R}^2)$ follows.
The companion operators in this paper obey homogeneous boundary conditions but are naturally extended to functions without any restriction. 
This leads to some  $b:= J b_\CR\in H^1(\Omega)$ with $b_\CR:= I_\CR b$. Given those functions, define the data 
$G:={\rm Curl } \: b \perp \nabla H^1_0(\Omega)$ and $g \equiv 0$ in \eqref{F:format}-\eqref{Fhat:format} for   $F$ and ${\widehat F}$. 
Then the exact solution   $u \equiv 0$ and hence the best approximation provides $u_\CR^{\rm mod}\equiv 0$ as well.
The investigation of the  original CRFEM solution  $u_\nc^{\rm org} \in  \CR_0^1(\T)$ to \eqref{eq:discrete} recalls the  discrete Helmholtz decomposition from Subsection~\ref{sec:cr}:
 $P_0(\T; {\mathbb R}^2)= \nabla_\pw \CR^1_0(\T)  \operp {\rm Curl}\: (S^1(\T) \setminus {\mathbb R})$. The $L^2$ orthogonality of
$\Pi_0 G= \Pi_0 \:  {\rm Curl } \: b = {\rm Curl }_\pw \: I_\CR b = {\rm Curl }_\pw $ $b_\CR \perp {\rm Curl } \:  S^1(\T)  $ 
  in $L^2(\Omega; {\mathbb R}^2)$ therefore proves $\Pi_0 G =\nabla_\pw u_\CR$
for some $u_\CR \in \CR^1_0(\T)$. In fact, $u_\CR$ solves the original CRFEM because of 
\[
a_\pw(u_\CR, v_\CR) = \int_\Omega \Pi_0 G \cdot \nabla_\pw v_\CR \dx = \int_\Omega G\cdot \nabla_\pw v_\CR \dx 
\] 
for all $v_\CR \in \CR^1_0(\T)$. That is, $\trinl u_\CR^{\rm org}\trinr_\pw = \trinl u_\CR \trinr_\pw = \|\Pi_0 G\|= \trinl I_\CR b \trinr_\pw = \trinl b_\CR \trinr_\pw=1$. 
This contradicts the best-approximation property.

\subsubsection{Counterexample to best-approximation of Morley FEM for $m=2$}
Suppose $|\T| \ge 2 $ and recall 
$S^1(\T) \subsetneq \CR^1_0(\T)$ for the existence of $\beta_\CR \in \CR^1(\T; {\mathbb R}^2) \setminus S^1(\T; {\mathbb R}^2)$. The linear Green strain $\varepsilon:= {\rm sym} \: D$ from linear elasticity allows for a Korn inequality with the rigid body motions 
$${\rm RH}(\Omega):=\{ \gamma \in P_1(\Omega; {\mathbb R}^2): D \gamma + D \gamma^T=0\} $$
(a linear subspace of dimension $3$ of the form $\gamma(x)- (a,b) + c(x_2,-x_1)$ for all $x \in \Omega$). The semi-scalar product $(\varepsilon(\bullet), \varepsilon(\bullet) )_{L^2(\Omega)}$ in the vector space 
$ S^1(\T; {\mathbb R}^2) / {\rm RH}(\Omega)$ is a scalar product (from the Korn inequality). 
The Riesz representation $\beta_{\rm c} \in  S^1(\T; {\mathbb R}^2) / {\rm RH}(\Omega)$ of the functional $v_{\rm c} \mapsto 
(\varepsilon_\pw (\beta_\CR), \varepsilon(v_{\rm c}))$ in $(S^1(\T; {\mathbb R}^2) / {\rm RH}(\Omega))^*$ provides $b_\CR:= \beta_\CR-\beta_{\rm c} \in 
\CR^1(\T; {\mathbb R}^2) \setminus S^1(\T; {\mathbb R}^2) $  with 
$\varepsilon_\pw(b_\CR) \perp \varepsilon (S^1(\T; {\mathbb R}^2)) $ in $L^2(\Omega; {\mathbb S})$. Without loss of generality suppose $\|\varepsilon (b_\CR)\|=1$. Recall the companion operator $J: \CR^1(\T) \rightarrow H^1(\Omega)$ from the previous example and apply 
$J b_\CR = (J b_\CR(1),  J b_\CR (2) ) \in H^1(\Omega; {\mathbb R}^2)$ componentwise to  $b_\CR \in \CR^1(\T; {\mathbb R}^2)$. 
Then let  $G:= {\rm sym} \: {\rm Curl} \: (J b_\CR(2), -J  b_\CR(1)) $ and $g \equiv 0$.  
Since $G \perp D^2 H^2_0(\Omega)$, the solutions  $u=0=u_\M^{\rm mod}$ coincide. 
On the other hand,  for all $\phi, \psi \in H^1(T; {\mathbb R}^2)$, 
\[ ({\rm sym} \: {\rm Curl} \:(\phi_2,-\phi_1))  :({\rm sym} \: {\rm Curl} \:(\psi_2,-\psi_1)) =
\varepsilon(\phi): \varepsilon(\psi) \text{ a.e. in } T. \]  
This relation shows 
\begin{align*}
\Pi_0 G & = {\rm sym} \: {\rm Curl}_\pw \:  I_\CR(J b_\CR(2),  -J b_\CR (1) )
= {\rm sym} \: {\rm Curl}_\pw \:  ( b_\CR(2),  - b_\CR (1) ) \\
 & \perp  {\rm sym} \: {\rm Curl}\;  S^1(\T; {\mathbb R}^2) \text{ in } L^2(\Omega;{\mathbb S}).  
 \end{align*}
The discrete Helmholtz decomposition
\(P_0(\T; {\mathbb S}) = D^2_\pw \M(\T) \operp  {\rm sym} \: {\rm Curl}\;  S^1(\T; {\mathbb R}^2/{\mathbb R}^2) \)
  from Subsection~\ref{sec:Morley} 
implies $\Pi_0 G =D^2_\pw u_\M$ for some $u_\M \in \M(\T)$. Then $u_\M=u_\M^{\rm org}$ and $\trinl u_\M^{\rm org} \trinl_\pw = \| \Pi_0 G\|= \|b_\CR\|_{L^2(T)}=1$. This contradicts the best-approximation property.

\section{A~posteriori analysis} \label{sec:aposteriori}
Unlike conforming schemes where the guaranteed error bound involves constants from approximation in the  error estimates of a quasi-interpolation operator, the constants in the nonconforming schemes appear small and are related to the simplicial element domain (and not  the shape of nodal patches).  Hence the resulting a posteriori error estimators can indeed be employed directly in the a~posteriori error control.

\subsection{Reliable error estimates}
\noindent The {\it natural right-hand side} in the discrete problem is $F_h:=\widehat{F}|_{V_{\nc}} \in V_\nc^*$  with   $\widehat{F}$ from \eqref{Fhat:format}.

\begin{thm}[a~posteriori estimate for $F_h=\widehat{F}|_{V_{\nc}}$] \label{thm_a:u-Junc} Suppose the $L^2$ orthogonality in  \eqref{eq:L2ortho}.  
Let $u \in V$  solve \eqref{eq:weakabstract} and let $u_\nc \in V_\nc$  solve \eqref{eq:discrete} with right-hand side $F_h:=\widehat{F}|_{V_{\nc}} $ from \eqref{Fhat:format}. 
Then 
\begin{align*}
(a) \quad 
\trinl  u-J u_\nc  \trinr^2 & + \trinl I_\nc u  - u_\nc  \trinr_\pw^2  \le  (\|G- \Pi_0 G \|+ 
\kappa_m  \| h_\T^m g\| +\trinl u_\nc- J u_\nc \trinr_\pw)^2, \\
(b)  \quad  \trinl u-u_\nc \trinr^2_\pw  & +\trinl I_\nc u-u_\nc \trinr^2_\pw  \le  (   \|G  - \Pi_0 G \| + \kappa_m \|h_\T^m g\| + \trinl u_\nc- J u_\nc \trinr_\pw)^2\\
& \qquad \qquad \qquad \qquad + 2 \widehat{F}(u_\nc-Ju_\nc).
\end{align*}
\end{thm}

The former analysis in  \cite{HuShi09, CCDGHU13, VeigaNiiranenStenberg07} is based on Helmholtz decompositions and worked out for simply-connected domains.
This is problematic and the three-dimensional application in  \cite{HuShi09} requires some 3D Helmholtz decompositions for a domain that is simply-connected and has a boundary that is connected and a discrete counterpart is unknown. The present analysis is rather simple and can immediately 
be extended to any bounded polyhedral Lipschitz domain with the help of the companion operator $J$ in 3D from \cite{CCP_new}. 

\medskip

\noindent{\it Proof of (a).}  For $e:= u-J u_\nc$,  elementary algebra and \eqref{eq:weakabstract}-\eqref{eq:best_approx}  lead to
\begin{align} \label{eq:rhs_a}
\trinl  e\trinr^2 &  = a(u,e)+ a_\pw(u_\nc- J u_\nc,e) -a_\pw(u_\nc,I_\nc e)  \notag\\
&= \widehat{F}(e- I_\nc e)+ a_\pw(u_\nc- J u_\nc,e-I_\nc e).
\end{align}
This, the definition of $\widehat{F}$ in \eqref{Fhat:format}, \eqref{eq:L2ortho}, and a Cauchy  inequality show 
\begin{align*}
\trinl  e\trinr^2     & \le  \int_{\Omega} (G-\Pi_0 G): D^m(e-  I_\nc e) \dx + \int_\Omega g(e-I_\nc e) \dx  \notag \\
& \quad +  \trinl u_\nc- J u_\nc  \trinr_\pw  \trinl e-  I_\nc e \trinr_\pw. 
\end{align*}
This, (weighted) Cauchy inequalities and
$
\| h_\T^{-m} (e-I_\nc e)\| \le \kappa_m \trinl e - I_\nc e\trinr_\pw
$ from \eqref{eq:interpolation_estimate} imply
\begin{align*}
\trinl  e\trinr^2     & \le  ( \|G-\Pi_0 G \| + \kappa_m \|h_\T^m g\| +  \trinl u_\nc- J u_\nc  \trinr_\pw)
\trinl e- I_\nc e \trinr_\pw.
\end{align*}
The Pythogoras theorem  \eqref{eq:Pythogoras} shows $\trinl  e\trinr^2 = \trinl  e- I_\nc e \trinr_\pw^2 + \trinl  I_\nc e\trinr_\pw^2$. This and Young's inequality conclude the proof. \hfill \qed

\medskip \noindent
{\it Proof of $(b)$.}  Elementary algebra and $a_\pw(u_\nc, u-I_\nc u)=0=a_\pw(I_\nc u, J u_\nc-u_\nc )$  from \eqref{eq:best_approx} result in 
\begin{equation*}
\trinl u-u_\nc \trinr^2_\pw= a_\pw(u -I_\nc u ,  Ju_\nc -u_
\nc) + a_\pw(u, u-Ju_\nc) - a_\pw(u_\nc, I_\nc u-u_\nc).
\end{equation*}
This, the Pythagoras theorem \eqref{eq:Pythogoras}, a Cauchy inequality,  the continuous problem \eqref{eq:weakabstract}  with the test function $u- J u_\nc$, and \eqref{eq:discrete}  with the test function $I_\nc u-  u_\nc$ lead to 
\begin{align} \label{eq:pyt_cauchya}
\trinl u-I_\nc u \trinr^2_\pw+\trinl I_\nc u-u_\nc \trinr^2_\pw & \le \trinl u-I_\nc u \trinr_\pw
\trinl u_\nc-J  u_\nc \trinr_\pw  \notag \\
& \quad + \widehat{F}(u- I _\nc u) +\widehat{F}(u_\nc-Ju_\nc).
\end{align}
 The definition of $\widehat{F}$ in \eqref{Fhat:format}, \eqref{eq:L2ortho}, (weighted) Cauchy inequalities, and \eqref{eq:interpolation_estimate} show
\begin{align*}
\widehat{F}(u- I _\nc u) & = \int_{\Omega} (G-\Pi_0 G): D^m(u-  I_\nc u) \dx + \int_\Omega g (u-I_\nc u) \dx  \\
& \le (\|G  - \Pi_0 G \| + \kappa_m \|h_\T^m g\|) \trinl u- I_\nc u  \trinr_\pw.  
\end{align*} 
The last two displayed  estimates  and Young's inequality conclude the proof. \qed

\medskip
\begin{rem} The   term $ \widehat{F}(u_\nc-Ju_\nc)$ can be computed and so $(b)$ is an a~posteriori error estimate. Moreover, \eqref{eq:companion_orthogonality} and (weighted) Cauchy inequalities show
$$\widehat{F}(u_\nc-Ju_\nc) \le (\|G  - \Pi_0 G \| + \kappa_m \osc_m(g,\T)) \trinl u_\nc- J u_\nc  \trinr_\pw. $$
The combination with \eqref{eq:pyt_cauchya} leads to (a) up to the extra factor 2. 
\end{rem}

\medskip 
\noindent The {\it modified choice of the right-hand side} reads $F_h=F \circ J$ for the right-inverse $J$ of $I_\nc$ from Section \ref{sec:companion}. 
The discrete problem for this choice of $F_h$ from \eqref{eq:modified_discrete} 
leads to a (new) a~posteriori estimate in Theorem \ref{thm:u-Junc}. 

\medskip

\begin{thm}[a posteriori estimate for  $F_h=F \circ J$] \label{thm:u-Junc} 
Suppose \eqref{F:format} and the $L^2$ orthogonality in \eqref{eq:L2ortho}. 
Let $u \in V$  solve \eqref{eq:weakabstract} and let $u_\nc \in V_\nc$  solve \eqref{eq:modified_discrete} with right-hand side $F_h:=F \circ J$. 
Then
\begin{align*}
(a) \; \trinl u-J u_\nc \trinr  & \le  \sqrt{1+ \Lambda_0^2} \|G- \Pi_0 G \|+ 
\sqrt{(\kappa_m  \| h_\T^m g\| +\trinl u_\nc- J u_\nc \trinr_\pw )^2 + 
\kappa_m^2 \Lambda_0^2 \osc^2_m(g,\T) }, \\
(b) \; \trinl u-u_\nc \trinr_\pw  & \le 
\sqrt{2(\trinl u_\nc-Ju_\nc \trinr^2 + {\rm apx}(F)^2)}
\end{align*}
with ${\rm apx}(F):= (1+ \Lambda_{\rm J} ) \|G - \Pi_0 G \| + \kappa_m \|h_\T^m g\|+\kappa_m \Lambda_{\rm J} \osc_m(g,\T)$ and $\Lambda_0$ (resp. $\Lambda_{\rm J}$) from \eqref{eq:lambda0} (resp.  \eqref{eq:companion_estimate}).
\end{thm}

\noindent{\it Proof of (a).} Elementary algebra with $e:= u-J u_\nc \in V$,  \eqref{eq:weakabstract}, $a_\pw(u_\nc, e-I_\nc e)~=0$ from \eqref{eq:best_approx}, $a_\pw(u_\nc-Ju_\nc, ~I_\nc e)~=~0$ from \eqref{eq:right_inverse}, and \eqref{eq:modified_discrete} lead to 
\begin{align} \label{eq:rhs}
\trinl  e\trinr^2 &  = a(u,e)+ a_\pw(u_\nc- J u_\nc,e) -a_\pw(u_\nc,e)  \notag\\
&= F(e- J I_\nc e)+ a_\pw(u_\nc- J u_\nc,e-I_\nc e). 
\end{align}
For any $v \in V$,   \eqref{eq:companion_orthogonality} and \eqref{eq:L2ortho} imply $\int_{\Omega} \Pi_0 G: D^m(v- J I_\nc v) \dx =0$. This, the definition  \eqref{F:format},  and  \eqref{eq:companion_orthogonality} result in 
\begin{align}\label{eq:F_error}
F(v- J I_\nc v)    & = \int_{\Omega} (G-\Pi_0 G): D^m(v- J I_\nc v) \dx + \int_\Omega g (v-I_\nc v) \dx  \notag \\
& \quad + \int_\Omega (g-\Pi_0 g)((1-J)I_\nc v) \dx \notag 
\\& \le \|G  - \Pi_0 G \|   \trinl v- J I_\nc v \trinr_\pw   + \|h_\T^m g\|
\|h_\T^{-m} (v- I_\nc v)\|  \notag \\
&\quad  + \osc_m(g,\T)
\|h_\T^{-m} (1 - J) I_\nc v \|
\end{align}
with (weighted) Cauchy inequalities in the last step. 
For 
 $\Lambda_0 = \|1-J\|_{L(V_\nc,\widehat{V})}$ from \eqref{eq:lambda0},   a triangle inequality  and $\trinl e\trinr_\pw^2 = \trinl e-I_\nc e \trinr_\pw^2 + \trinl I_\nc e\trinr_\pw^2$ from \eqref{eq:Pythogoras} prove
\[  \trinl e- J I_\nc e \trinr_\pw \le  \trinl e-  I_\nc e \trinr_\pw+ \trinl (1- J) I_\nc e \trinr_\pw \le \sqrt{1+ \Lambda_0^2} \trinl e \trinr_\pw.
\]
Recall from \eqref{eq:interpolation_estimate} and \eqref{eq:Pythogoras} that 
\[ \| h_\T^{-m} (e-I_\nc e)\| \le \kappa_m \trinl e - I_\nc e\trinr_\pw   \text{ and }
\| h_\T^{-m} (1-J)  I_\nc e\|\le \kappa_m \Lambda_0 \trinl  I_\nc e \trinr_\pw.
\]
 Substitute the last three displayed inequalities in \eqref{eq:F_error} with $v:=e$ and then substitute the resulting estimate in \eqref{eq:rhs}. Apply a Cauchy inequality for the second term in the right-hand side of \eqref{eq:rhs}  and  use \eqref{eq:Pythogoras} to conclude the proof of $(a)$. \hfill \qed

\medskip \noindent 
{\it Proof of (b).} Elementary algebra with $a_\pw(u_\nc, u-I_\nc u)=0=a_\pw(I_\nc u, J u_\nc-u_\nc )$ from  \eqref{eq:best_approx} and the continuous problem \eqref{eq:weakabstract} with test function $u- J u_\nc$ show
\[ 
\trinl u-u_\nc \trinr^2_\pw= F(u- Ju_\nc) + a_\pw(u -I_\nc u, Ju_\nc - u_\nc) - a_\pw(u_\nc, I_\nc u-u_\nc).
\]
The Pythagoras theorem \eqref{eq:Pythogoras}, a Cauchy inequality, and \eqref{eq:modified_discrete} reveal in the last identity that 
\begin{equation} \label{eq:pyt_cauchy}
\trinl u-I_\nc u \trinr^2_\pw+\trinl I_\nc u-u_\nc \trinr^2_\pw \le \trinl u-I_\nc u \trinr_\pw
\trinl u_\nc-J  u_\nc \trinr_\pw + F(u- JI _\nc u).
\end{equation}

\noindent The choice of  $v:=u$ in \eqref{eq:F_error}, a triangle inequality, \eqref{eq:interpolation_estimate}, and twice \eqref{eq:companion_estimate}  for $\trinl I_\nc u - J I_\nc u \trinr_\pw \le \Lambda_{\rm J} \trinl u-I_\nc u \trinr_\pw$ lead to 
\begin{align}\label{eq:F}
& F(u- J I_\nc u)  \le \trinl u - I_\nc u\trinr_\pw ( (1+ \Lambda_{\rm J}) \|G- \Pi_0 G \| + \kappa_m  \| h_\T^m g\|  \notag \\
&\qquad  \quad \qquad  \quad + \kappa_m \Lambda_{\rm J} \osc_m(g, \T) ) = {\rm apx} (F) \trinl u - I_\nc u\trinr_\pw. 
\end{align}
A substitution of \eqref{eq:F} in \eqref{eq:pyt_cauchy} and Young's inequality show 
\begin{equation*}
\frac{1}{2} \trinl u - I_\nc u\trinr^2_\pw + \trinl  I_\nc u - u_\nc \trinr^2_\pw \le 
  \trinl u_\nc- J u_\nc \trinr^2_\pw +    {\rm apx}(F)^2.
\end{equation*}
 This concludes the proof.  \qed
 
\begin{rem}[generalization]
 The model scenario of the right-hand side in  \eqref{F:format} solely coincides with the lowest and highest order terms in \eqref{eq:formf} 
 because oscillations occur in those terms because of \eqref{eq:best_approx} and \eqref{eq:companion_orthogonality}. The generalization of Theorems \ref{thm_a:u-Junc}-\ref{thm:u-Junc} 
 to the general situation \eqref{eq:formf}  is straightforward and the extra terms 
read $\displaystyle \sum_{1 \le |\alpha| \le m-1 } \|h_\T^{m-|\alpha|} f_\alpha \|_{L^2(\Omega)}.$
\end{rem}

\begin{rem}[a posteriori estimates for \eqref{modf}]
 For the right-hand side in \eqref{modf},  the discretisation will change because $\widehat{F}$ changes to 
 \[
\widehat{F}(\widehat{v}):= \int_{\Omega} (G- Q): D_\pw^m \widehat{v}\dx + \int_{\Omega} (g+ {\rm div}^m\: Q) \widehat{v} \dx   \text{ for all } \widehat{v} =v+v_\nc \in \widehat{V} \equiv V+V_\nc.
\]
  It may be seen as an advantage of the modified scheme \eqref{eq:modified_discrete} that $u_\nc$ is not affected by a change of $(G,g)$ to $(G{-} Q, g +{\rm div}^m \: Q )$ for any $Q$. 
The a~posteriori error control involves the terms $\|G- \Pi_0 G\|+ \| h^m_\T g\|$ and this motivates the minimisation 
\begin{align}\label{eq:min}
\min_{Q \in H({\rm div}^m, Q)} \left( \|(1-\Pi_0)(G-Q) \|^m +\kappa_m \|h_\T^m (g+ {\rm div}^m \:  Q) \| \right).
\end{align}
 The modified data $(G',g')=(G{-}Q, g+ {\rm div}^m \: Q)$ for a minimizer $Q$ leads to optimal a~posteriori error control in Theorem~\ref{thm:u-Junc} when the upper bound is simplified to 
 \begin{align}\label{eq:simplify}
 \frac{1}{2} \trinl u-u_\nc \trinr^2_\pw \le \trinl u_\nc- J u_\nc \trinr_\pw^2  +
2 (1+ \Lambda_{\rm J})^2 \left(\|G'- \Pi_0 G'\| + \kappa_m \|h_\T^m g'\| \right)^2.
 \end{align}
The situation is less clear in Theorem~\ref{thm_a:u-Junc} because $u_\nc$ may be affected by a change of $\widehat{F} $ and this may affect $\trinl u_\nc - J u_\nc \trinr_\pw$. However, it is unrealistic to solve the minimisation problem \eqref{eq:min} exactly for general $G$ and $g$. Nevertheless, a preprocessing of the data $G$ and $g$ can reduce the data approximation terms in (guaranteed) upper error bounds. In return, the data selected can be poor as well and ruin the efficiency of the reliable error estimates in Theorems~\ref{thm_a:u-Junc} and \ref{thm:u-Junc}.
\end{rem}

\subsection{Efficiency up to data-oscillations}
The efficiency of the term $\|h_\T^m g \|_{L^2(\Omega)} $ up to data-oscillations is established next.
\begin{thm}[efficiency up to data-oscillations] Let $u$ solve \eqref{eq:weakabstract} with right-hand side \eqref{F:format} for $G,g \in L^2(\Omega)$. Then 
$\displaystyle \|h_\T^m g \|_{L^2(\Omega)} \lesssim \trinl u-I_\nc u \trinr_\pw + \osc_m(g,\T) + \|G-\Pi_0 G \|.$
\end{thm}

\begin{proof} 
The efficiency holds in a local form for each triangle $T \in \T$. Let $\varphi_z \in S^1(\T)$ be the $P_1$ nodal basis function associated to a vertex $z \in \V(T) =\{a,b,c\}$ of the triangle $T$, i.e., $\varphi_z|_T$ is a barycentric coordinate of $z$ in $T$ and let $b_T:= 27 \varphi_a \varphi_b \varphi_c \in P_3(T) \cap H^1_0(T)$ be the cubic-bubble function. Then $b_T^m \in H^m_0(T) \subset V$ satisfies $0 \le b_T^m \le 1$ and inverse estimates $\trinl b_T^m \trinr \lesssim h^{-m}_T \|b_T\|$. Abbreviate $g_m:= \Pi_m g|_T \in P_m(T) $ and observe
\begin{align} \label{eq:gm}
\|h_T^m g\|_{L^2(T)} \le \|h_T^m g_m\|_{L^2(T)}  + \osc_m(g,T )
\end{align}
from the definition of $\osc_m(g,T)$ and a triangle inequality.  An inverse estimate $\| g_m\|_{L^2(T)} \le \constant{} \label{gg}$
$\|b_T^{m/2} g_m \|_{L^2(T)}$ follows from the equivalence of the norms 
$\|b_T\bullet\|_{L^2(T)} \approx \|\bullet \|_{L^2(T)}$  in $P_m(T)$. (A scaling argument  reveals that the equivalence constant does not depend on the size or shape of the triangles and exclusively depends on $m$.) The inverse estimate shows
\begin{align}\label{eq:invest}
\constant[\ref{gg}]^{-2} h_T^{2m} \|g_m\|_{L^2(T)}^2 & \le h_T^{2m} \|b_T^{m/2} g_m\|^2_{L^2(T)}=h_T^{2m} \int_T b_T^m g_m(g_m-g) \dx+ 
h_T^{2m} \int_T b_T^m g_m g \dx \nonumber \\
& \le \osc_m(g,\T) h_T^m \|b_T^{m} g_m \|_{L^2(T)} + h_T^m  \int_T g v_T  \dx
\end{align}
with a Cauchy inequality and the abbreviation $v_T:= h_T^m b_T^m g_m \in P_{4m}(T) \cap H^2_0(T) \subset V$ in the last step. 
Since $a(u,v_T)=\int_Tg v_T \dx +\int_T G: D^m v_T \dx $ (from \eqref{eq:weakabstract} for the test function $v_T\in V$), the crucial term in  \eqref{eq:invest} is equal to 
\begin{equation} \label{eq:crucial}
h_T^m \int_T g v_T \dx =h_T^m \int_T(D^2 u - G) : D^m v_T  \dx.
\end{equation}
The bubble-function methodology due to \cite{Verfurth13}  observes at this point that $\Pi_0 D^m v_T=0$ from an integration by parts with $v_T \in H^m_0(T)$, 
whence $D^m v_T=0$ on $\partial T$, and so $ \int_T D^m v_T \dx = \int_{\partial T} (D^{m-1} v_T) \nu_T \ds =0.$
Consequently, 
\begin{align*}
\int_T (D^m u-G): D^m v_T \dx & =\int_T D^m v_T: (1-\Pi_0) (D^m u -G)\dx \\
& \le ({ |u-I_\nc u|_{H^2(T)} } + \|G - \Pi_0 G \|_{L^2(T)}) \|D^m v_T\|_{L^2(T)}.
\end{align*}
 A standard inverse estimate $\|D^m v_T\|_{L^2(T)} \lesssim h_T^{-m}\|v_T\|_{L^2(T)}$ for the polynomial $v_T \in P_{4m}(T)$ depends on $m$ and on the shape-regularity  of the triangle $T$. Note that $h_T^{-m} \|v_T\|_{L^2(T)}=\|b_T^m g_m\|_{L^2(T)} \le \|g_m\|_{L^2(T)} $ follows from the definition of $v_T= h_T^{m}b_T^m g_m$ and $0 \le b_T \le 1$. The aforementioned estimates lead in \eqref{eq:crucial} to 
\[ h_T^m \int_T g v_T \dx \le \constant{}\label{ss} (|u-I_\nc u|_{H^2(T)}  + \|G - \Pi_0 G \|_{L^2(T)}) h_T^m \|g_m\|_{L^2(T)} \]
for some generic constant  $\constant[\ref{ss}] \approx 1$ that depends on the shape of $T$ and $m$.  The combination of the previous estimate with \eqref{eq:invest} proves
\[ h_T^{2m} \|g_m\|^2_{L^2(T)} \le \constant[\ref{gg}] h_T^m \|g_m\|_{L^2(T)} 
(\osc_m(g,T) + \constant[\ref{ss}]  (|u-I_\nc u|_{H^2(T)}  + \|G - \Pi_0 G \|_{L^2(T)})),  \]
whence
\( h_T^{m} \|g_m\|_{L^2(T)} \lesssim |u-I_\nc u|_{H^2(T)}+\|G - \Pi_0 G \|_{L^2(T)} +\osc_m(g,T). \)
This and \eqref{eq:gm} lead to the asserted local efficiency 
\( \|h_T^m g\|^2_{L^2(T)} \lesssim  |u-I_\nc u|^2_{H^2(T)}+\|G - \Pi_0 G \|^2_{L^2(T)} +\osc_m(g,T)^2\)
for each $T \in \T$. The sum over all $T \in \T$ concludes the proof.
\end{proof}

\subsection{Example with dominating data oscillations}
The data-oscillation term $\|G-\Pi_0 G\|$ dominates the error zero and makes the error estimates useless in the following alarming example.
Suppose $m=2$ (the situation is simpler for $m=1$) and choose some function $z_{\rm c} \in S^1(\T; {\mathbb R}^2)$. 
Given $z_{\rm c}$, consider any $z \in H^1(\Omega; {\mathbb R}^2)$ with 
$\int_E z \ds=\int_E z_{\rm c} \ds$ for all edges $E \in \E$. 
(This can be done, e.g, by volume bubble functions added to $z_{\rm c}$ or other functions that do not change the above edge integrals.)
Let $g \equiv 0$ and $G:=  {\rm sym \; Curl} \:z \in L^2(\Omega; {\mathbb S})$ in \eqref{F:format}. 
Since {$\Pi_0 \: {\rm Curl } \: z= {\rm Curl }_\pw \: I_\CR z = {\rm Curl } \: z_{\rm c},$} it follows that 
$\int_{\Omega} G : D^2 v \dx =0 =\int_{\Omega} G : D_\pw^2 v_\M \dx $ for all $v \in V$ and $v_\M \in \M(\T)$.  
The first identity follows immediately from an integration by parts and the second from a piecewise integration by parts; for instance,  any $v_\M \in \M(\T)$ satisfies
\begin{align*}
\int_{\Omega} G : D_\pw^2 v_\M \dx &= \int_{\Omega} {\rm sym \;  Curl} \:z :  D^2_\pw v_\M \dx
= \int_{\Omega} {\rm Curl} \:z : D^2_\pw v_\M \dx  \\
&= \int_{\Omega} {\rm Curl} \:z_{\rm c} : D^2_\pw v_\M \dx
 = \sum_{E \in \E} \int_E 
 \nu_E \cdot \langle {\rm Curl} \:z_{\rm c} \rangle_E [\nabla_\pw v_\M]_E \ds \\
& =\sum_{E \in \E} \left( \int_E  [\nabla_\pw v_\M]_E  \ds \right) \cdot \langle {\rm Curl} \:z_{\rm c} \rangle_E \nu_E \ds =0
\end{align*}
(because $[{\rm Curl} \:z_{\rm c}]_E \times \nu_E =0$ and $ \langle {\rm Curl} \:z_{\rm c} \rangle_E \nu_E =0$ on $E$
while $\int_E [\nabla_\pw v_\M]_E \ds =0$). This latter orthogonality  property implies $u=0=u_\M$ for the scheme \eqref{eq:discrete} as well as for \eqref{eq:modified_discrete}. 
On the other hand, 
$\|G- \Pi_0 G\|= \|{\rm sym \;  Curl} \;(z -z_{\rm c})\|$. If $\|G-\Pi_0 G\|_{L^2(T)} =0$ for $T \in \T$, then $w=(w_1,w_2):=z-z_{\rm c} \in H^1(T;{\mathbb R}^2)$ satisfies 
$\displaystyle  {\rm sym \;  Curl} \: w =  0$ 
a.e. in  $T$. This is equivalent to $w_{1,2}=0=w_{2,1}$ and $w_{1,1}=w_{2,2}$. 
Hence $(w_2, w_1)$ is divergence-free and there exists some $\beta \in H^2(T)$ with 
$(w_2, w_1) = {\rm \nabla}\beta = (\partial \beta/\partial x_1, \partial \beta/\partial x_2).$ Consequently, $\beta$ satisfies $ \partial^2 \beta/\partial x_j^2 =0$ for $j=1,2$.  Since $\beta$ is harmonic, it is smooth in $T$. Without loss of generality, assume $0 \in {\rm int} \: (T)$. For $x=(x_1,x_2) \in T$, an integration along the line
 ${\rm conv} \{0,{x_1}\} \times \{x_2\}$ 
 shows $\beta(x)=\beta(0,x_2)+\int_0^{ x_1} \beta_{,1}(\xi,x_2) \: {\rm d}\xi$ (with the abbreviation $\beta_{,j}:=\partial \beta/\partial x_j$, $\beta_{,ij}:=\partial^2 \beta/\partial x_i \partial x_j$, $j=1,2$
for the partial derivatives). Since $\beta_{,11}=0$, $\beta_{,1}(\xi,x_2)=\beta_{,1}(0, x_2)$ for a.e. $\xi \in {\rm conv} \{0,x_1\} $. Hence $\beta(x)=\beta(0,x_2) +x_1 \beta_{,1}(0,x_2)$. The same argument applies for an integration along $\{x_1\} \times {\rm conv} \{0,{x_2}\} $ and leads to $\beta(x)=\beta(x_1,0) +x_2 \beta_{,2}(x_1,0)$. Recall that $\beta$ is smooth in  ${\rm int} \: (T)$ and consider derivatives of the two identities to verify $\beta_{,12}(x)= \beta_{,12}(0,x_2)=\beta_{,12}(x_1,0)$  for all $x \in {\rm int} \: (T)$. The latter identity shows first for $(0,x_2) \in T$ that $\beta_{,12}(0,x_2)=\beta_{,12}(0,0)$ and second for all $x \in {\rm int} \: (T)$ that $\beta_{,12}(x)=\beta_{,12}(0,0)$. Consequently, the Hessian $D^2 \beta$ of $\beta$ is constant as $\beta \in P_2(T)$. Recall that $\int_E w \ds =0 =\int_E \nabla \beta \ds$ for all edges $E \in \E(T)$ of $T$. Therefore the affine vector field $\nabla \beta \in P_1(T; {\mathbb R}^2)$
vanishes at all the edges, midpoints and so everywhere in $T$. Consequently, $\beta$ is a constant and $w:=(w_1, w_2)= (\beta_{,2},\beta_{,1}) =0$ vanishes a.e. in $T$. This shows that any choice of $z \in H^1(\Omega; {\mathbb R}^2)$ with $\|z-z_{\rm c}\|_{L^2(\Omega)}>0$ leads to $\|G- \Pi_0 G\|>0$.
In other words, the error vanishes while the data approximation does not.

\section*{Acknowledgements}
The research of the first author has been supported by the Deutsche Forschungsgemeinschaft in the Priority Program 1748 under the project "foundation and application of generalized mixed FEM towards nonlinear problems
in solid mechanics" (CA 151/22-2).  The finalization of this paper has been supported by   SPARC project 
(id 235) entitled {\it the mathematics and computation of plates}.

\bibliographystyle{amsplain}
\bibliography{ref_combined}
\end{document}